\documentclass[a4paper,twoside]{article}

\usepackage[utf8]{inputenc}
\usepackage[T1]{fontenc}
\usepackage[margin=2.5cm]{geometry}

\usepackage[dvipsnames]{xcolor}
\usepackage{hyperref}
\hypersetup{colorlinks,linkcolor={red},citecolor={blue},urlcolor={green!70!black}}
\usepackage{enumitem}

\usepackage{amsfonts,amssymb,bm}
\usepackage{siunitx}
\usepackage{float}

\usepackage{amsmath,empheq}
\numberwithin{equation}{section}

\usepackage{amsthm}
\newtheoremstyle{mytheoremstyle}{7pt}{7pt}{\normalfont}{}{\normalfont\bfseries}{:}{.5em}{}
\theoremstyle{mytheoremstyle}
\newtheorem{definition}{Definition}[section]

\newtheorem{proposition}[definition]{Proposition}
\newtheorem{theorem}[definition]{Theorem}
\newtheorem{corollary}[definition]{Corollary}

\newtheorem{remark}[definition]{Remark}

\usepackage[square,authoryear]{natbib}

\newcommand{\SetFont}[1]{\mathbb{#1}}

\newcommand{\R}{\SetFont{R}}
\newcommand{\LieGroupFont}[1]{\mathsf{#1}}

\newcommand{\SO}{\LieGroupFont{SO}}

\newcommand{\LieAlgebraFont}[1]{\mathfrak{#1}}

\newcommand{\so}{\LieAlgebraFont{so}}

\newcommand{\dif}{\bm{\mathrm{d}}}

\newcommand{\ad}{\mathrm{ad}}
\newcommand{\Ad}{\mathrm{Ad}}
\newcommand{\Id}{\mathrm{Id}}
\newcommand{\Lag}{L}
\newcommand{\lag}{\ell}

\newcommand{\Tr}{\operatorname{Tr}}
\newcommand{\Orb}{\operatorname{Orb}}

\begin{document}

\title{\textbf{Variational discretization of thermodynamical simple systems on Lie groups}}
\author{Benjamin \textsc{Cou\'eraud}\\ CNRS -- LMD -- IPSL\\ \'Ecole Normale Sup\'erieure de Paris -- PSL\\ 24 rue Lhomond, 75005 Paris, France\\ \href{mailto:coueraud@lmd.ens.fr}{\texttt{coueraud@lmd.ens.fr}}\and François \textsc{Gay-Balmaz}\\ CNRS -- LMD -- IPSL\\ \'Ecole Normale Sup\'erieure de Paris -- PSL\\ 24 rue Lhomond, 75005 Paris, France\\ \href{mailto:francois.gay-balmaz@lmd.ens.fr}{\texttt{francois.gay-balmaz@lmd.ens.fr}}}
\date{\it In Honour of J\"urgen Scheurle's 65th Birthday}

\maketitle

\begin{abstract}
This paper presents the continuous and discrete variational formulations of simple thermodynamical systems whose configuration space is a (finite dimensional) Lie group. We follow the variational approach to nonequilibrium thermodynamics developed in \cite{GBYo2017a,GBYo2017b}, as well as its discrete counterpart whose foundations have been laid in \cite{GBYo2017c}. In a first part, starting from this variational formalism on the Lie group, we perform an Euler-Poincar\'e reduction in order to obtain the reduced evolution equations of the system on the Lie algebra of the configuration space. We obtain as corollaries the energy balance and a Kelvin-Noether theorem. In a second part, a compatible discretization is developed resulting in discrete evolution equations that take place on the Lie group. Then, these discrete equations are transported onto the Lie algebra of the configuration space with the help of a group difference map. Finally we illustrate our framework with a heavy top immersed in a viscous fluid modeled by a Stokes flow and proceed with a numerical simulation.
\end{abstract}

\tableofcontents

\section*{Introduction}
\addcontentsline{toc}{section}{Introduction}

In \cite{GBYo2017a} and \cite{GBYo2017b}, a new variational formalism was proposed for nonequilibrium thermodynamics. This formalism is an extension of the Hamilton principle that allows the inclusion of irreversible phenomena in discrete and continuum systems, by using a nonholonomic nonlinear constraint, the so-called \emph{phenomenological constraint}, and its associated \emph{variational constraint}. In this formalism, the entropy of the system is promoted to a full dynamic variable, and to each irreversible process corresponds a \emph{thermodynamic displacement}, whose rate equals the \emph{thermodynamic affinity} of the process. Thanks to the introduction of these variables, together with the phenomenological and variational constraints, this new variational formalism yields the time evolution equations of the system in accordance with the two fundamental laws of thermodynamics, see, e.g., \citet[Chapter 1]{St1974}.

Equipped with such a variational formalism, it is natural to try to devise new \emph{variational integrators} based on this approach, with the aim of developing new algorithms for the simulation of multiphysics systems with the advantages that are known to variational integrators of Lagrangian mechanics, see \cite{MaWe2001}. Indeed, variational integrators were proved to be superior to more classical algorithms thanks to the fact that they were designed to preserve as much as possible the geometric structures underlying the mechanical system they discretize. Some important features are that the discrete symplectic structure of the discrete Lagrangian system, as well as the discrete momenta in case where there are symmetries, are all preserved, and the discrete total energy of the system remains bounded and oscillates around its correct value during the simulation. In \cite{GBYo2017c}, variational integrators for the nonequilibrium thermodynamics of simple closed systems were developed. Because of the presence of thermal effects, the flow of the continuous equations is not symplectic anymore, but rather satisfies a generalized \emph{structure-preserving property}, which reduces to the conservation of the symplectic form if thermal effects are absent. The associated variational integrators satisfy a discrete version of this structure-preserving property.

In this paper, we are still mainly concerned with simple closed systems, but whose configuration space is a Lie group $G$. Such systems may possess symmetries, leading to the natural idea of performing an Euler-Poincar\'e reduction to obtain the reduced time evolution equations of the system on the associated Lie algebra $\LieAlgebraFont{g}$, and discretize the reduced variational principle accordingly. The paper is organized as follows. 

In the first part, starting from the thermodynamical variational principle presented in \cite{GBYo2017a} for a simple system with configuration space a (finite-dimensional) Lie group, we perform an Euler-Poincar\'e reduction (Theorem \ref{thm:TEP-reduction}), in the general case where there is an advected parameter taking values in a arbitrary manifold as in \cite{GBTr2010}. In this setting, the group $G$ \emph{does not act} on the entropy variable $S$, which is a natural assumption for simple systems. As corollaries we obtain the (reduced) energy balance (Corollary \ref{cor:energy-balance}) and an abstract generalization of the well-known Kelvin-Noether of fluid dynamics (Corollary \ref{cor:Kelvin-Noether}). Compared to the usual Euler-Poincar\'e case, the presence of thermal effects is responsible for the presence of additional terms.

The second part deals with the variational discretization of such systems from the point of view of \cite{GBYo2017c} for thermodynamical systems on one hand, and of \cite{MaPeSh1999} for Euler-Poincar\'e systems on the other hand. In Section \ref{sec:DTEP-reduction} we follow \cite{MaPeSh1999} to obtain a discrete Euler-Poincar\'e reduction theorem (Theorem \ref{thm:DTEP-reduction}) for simple thermodynamical systems. We then transport these equations onto the Lie algebra $\LieAlgebraFont{g}$ with the help of a group difference map, as introduced in \citet[Section 4]{BRMa2008}, obtaining in this way an integrator on the Lie algebra $\LieAlgebraFont{g}$ instead of the Lie group $G$, which is more practical from the computational point of view.

Finally we illustrate our framework with one example: a heavy top immersed in a viscous fluid whose flow is approximated by a Stokes flow for practical purposes. This example naturally involves the Lie group $\LieGroupFont{SO}(3)$ of rotations of the Euclidean space $\R^3$, which is the configuration space of the rigid body. The viscous fluid represents the direct physical environment of the rigid body, however, it is actually part of the system we consider, in a way which is similar to the example presented in \citet[Section 5]{GBYo2017c}. Indeed, in the model we will build, we only need one entropy variable to describe the thermodynamics of both the rigid body and the fluid. We write down the continuous system and its variational discretization, which yields an integrator with an interesting energy behavior, as confirmed by our numerical simulations. {This example only illustrates a simplified situation of the general variational setting that we develop in this paper. We postpone the applications of the general setting for future works in the context of fluid thermodynamics}.

\section{Euler-Poincar\'e reduction for simple thermodynamical systems}

In this section we shall present a Lagrangian reduction process for simple thermodynamical systems with symmetries, by focusing on the case when the configuration manifold of the mechanical variables of the system is a Lie group. This process extends to simple thermodynamical systems the well-known process of Euler-Poincar\'e reduction for mechanical systems on Lie groups.

\subsection{Variational formulation of thermodynamics}\label{sec:TVP}

We first review the main points of the variational formalism for nonequilibrium thermodynamics of simple systems that has been introduced in \cite{GBYo2017a}. Recall that by a \emph{simple system} we mean a thermodynamical system for which we only need one entropy variable $S$ and a finite set of mechanical variables $(q,\dot{q})$ in order to describe entirely the state of the system. Moreover, we will assume in the following that such systems don't exchange matter with their environment, that is, they are \emph{closed}. Given such a closed simple system, let $Q$ be the configuration manifold associated to the mechanical variables of the system, assumed to be finite dimensional. The Lagrangian of such a system depends on the position, velocity and entropy of the system, and is therefore a map $\Lag:TQ\times\R\to\R$, $(q, \dot q, S) \mapsto \Lag(q, \dot q, S)$. We also have forces that act on the system: external forces $F^\text{ext}:TQ\times\R\to T^*Q$ that do not derive from a potential and friction forces $F^\text{fr}:TQ\times\R\to T^*Q$ which ultimately encode all the irreversible processes in the simple system and are responsible for internal entropy production. 
These maps are assumed to be fiber preserving, that is, $F^\text{ext}(q,\dot{q},S)$, $F^\text{fr}(q,\dot{q},S)\in T_q^*Q$ for any $q\in Q$, $\dot{q}\in T_q Q$ and $S\in\R$. It is also possible that the system exchanges heat with its environment, and we will denote by $P^\text{ext}_H:TQ\times\R\to\R$ the power due to heat transfer with the exterior of the system. Note that both the external force and the external heat can also depend explicitly on time; however, for simplicity, this will not be the case in this article. 
Now that the various entities have been set up, we can state the variational formulation for nonequilibrium thermodynamics of simple closed systems \citet[Section 3]{GBYo2017a}.

A curve $(q,S)$ satisfies the \emph{variational formulation for nonequilibrium thermodynamics} if and only if it satisfies the \emph{variational condition}
\begin{equation}\label{eq:variational-condition}
	\delta\int_0^T\Lag(q,\dot{q},S)\,\mathrm{d}t+\int_0^T\big\langle F^\text{ext}(q,\dot{q},S),\delta q\big\rangle\,\mathrm{d}t=0,
\end{equation}
for all variations $\delta q$ and $\delta S$ satisfying the \emph{variational constraint}
\begin{equation}\label{eq:variational-constraint}
	\frac{\partial\Lag}{\partial S}(q,\dot{q},S)\delta S=\left\langle F^\text{fr}(q,\dot{q},S),\delta q\right\rangle,
\end{equation}
with $ \delta q(0)= \delta q(T)=0$ and if it also satisfies the \emph{phenomenological constraint}
\begin{equation}\label{eq:phenomenological-constraint}
	\frac{\partial\Lag}{\partial S}(q,\dot{q},S)\dot{S}=\left\langle F^\text{fr}(q,\dot{q},S),\dot{q}\right\rangle-P^\text{ext}_H(q,\dot{q},S).
\end{equation}

Note that this variational formulation is an extension of Hamilton's principle of classical mechanics to the thermodynamics of simple systems. The constraint \eqref{eq:phenomenological-constraint} on the curve $(q,S)$ is nonlinear and nonholonomic in $\dot{q}$. The name of \eqref{eq:phenomenological-constraint} comes from the fact that friction forces involve phenomenological laws, which pertain nonequilibrium thermodynamics. The constraint \eqref{eq:variational-constraint} on the variations $ ( \delta q, \delta S)$  follows from  \eqref{eq:phenomenological-constraint} by formally replacing the velocity by the corresponding virtual displacement, and by removing the contribution from the exterior of
the system. Such a simple correspondence between the phenomenological and variational constraints still holds for more
general thermodynamic systems, see \cite{GBYo2017a}.
Taking variations in the variational condition \eqref{eq:variational-condition} and using the constraints \eqref{eq:variational-constraint} and \eqref{eq:phenomenological-constraint}, we obtain the following system of differential equations:
\begin{empheq}[left=\empheqlbrace]{align}
		&\frac{\mathrm{d}}{\mathrm{d}t}\frac{\partial\Lag}{\partial\dot{q}}(q,\dot{q},S)-\frac{\partial\Lag}{\partial q}(q,\dot{q},S)=F^\text{ext}(q,\dot{q},S)+F^\text{fr}(q,\dot{q},S),\label{eq:mechanical-equation}\\[2mm]
		&\frac{\partial\Lag}{\partial S}(q,\dot{q},S)\dot{S}=\left\langle F^\text{fr}(q,\dot{q},S),\dot{q}\right\rangle-P_H^\text{ext}(q,\dot{q},S).\label{eq:thermodynamical-equation}
\end{empheq}

Introducing the \emph{temperature} $T(q,\dot{q},S)=-\frac{\partial\Lag}{\partial S}(q,\dot{q},S)$, which is assumed to be positive, the second equation reads:
\[
	\dot{S}=-\frac{1}{T}\left\langle F^\text{fr}(q,\dot{q},S),\dot{q}\right\rangle+\frac{1}{T}P^\text{ext}_H(q,\dot{q},S),
\]
whose first term is interpreted as the \emph{internal entropy production} of the simple closed system. In case the system is adiabatically closed, there is no heat nor matter transfer with the environment, therefore it is necessary to have $\left\langle F^\text{fr}(q,\dot{q},S),\dot{q}\right\rangle\leq 0$ for any $(q,\dot{q},S)\in TQ\times\R$, meaning that friction forces are \emph{dissipative}, in order for this equation to agree with the second law of thermodynamics, which states that entropy is always increasing for an adiabatically closed system.

The energy associated with the Lagrangian $\Lag$ is the function $E:TQ\times\R\to\R$ defined by
\begin{equation}\label{def:energy}
	E(q,\dot{q},S)=\left\langle\frac{\partial\Lag}{\partial q}(q,\dot{q},S),\dot{q}\right\rangle-\Lag(q,\dot{q},S),
\end{equation}
for all $(q,\dot{q},S)\in TQ\times\R$. Using \eqref{eq:mechanical-equation} and \eqref{eq:thermodynamical-equation}, and defining the power of the external forces as $P_W^\text{ext}(q,\dot{q},S)=\left\langle F^\text{ext}(q,\dot{q},S),\dot{q}\right\rangle$, we recover the general \emph{energy balance}, that is,
\begin{equation}\label{eq:energy-balance}
	\frac{\mathrm{d}E}{\mathrm{d}t}=P_W^\text{ext}(q,\dot{q},S)+P_H^\text{ext}(q,\dot{q},S),
\end{equation}
along a solution $(q,S)$ of the thermodynamical Euler-Lagrange equations. Thus we recover the first law of thermodynamics.

In a nutshell, the variational formalism reviewed above yields the time evolution equations for the thermomechanical system considered, in accordance with the axiomatic formulation of thermodynamics of Stueckelberg (\citet[Chapter 1]{St1974}). We refer the reader to \cite{GBYo2017a,GBYo2017b} for more details on this formalism, historical background, as well as numerous examples.

\subsection{Euler-Poincar\'e reduction}\label{sec:TEP-reduction}

In this section we consider simple thermodynamical systems on Lie groups, meaning that the configuration space $Q$ of the mechanical part of the system is a finite dimensional Lie group $G$. Given $g\in G$, denote by $L_g$ the left multiplication by $g$ in the group $G$, $\LieAlgebraFont{g}$ its Lie algebra, and $\omega_G\in\Omega^1(G,\mathfrak{g})$ its left Maurer-Cartan form \citet[Chapter 3, Definition 1.3]{Sh1997}. Recall that $\omega_G$ is the $\mathfrak{g}  $-valued one-form on $G$ defined by $\omega_G(\dot{g})=T_g L_{g^{-1}}(\dot{g})\in T_eG\cong\LieAlgebraFont{g}$ for any $\dot{g}\in T_gG$. The left action of $G$ on itself can be lifted to $TG$, and we get $TG/G\cong\LieAlgebraFont{g}$, where the diffeomorphism is given by the Maurer-Cartan form $\omega_G$.

Given a $G$-invariant Lagrangian $\Lag:TG\to\R$ on $G$, and the associated Euler-Lagrange equations, it is natural to ask how one can obtain equivalent equations directly on the Lie algebra $\mathfrak{g}$, which is the realization of the \emph{reduced} velocity phase space $TG/G$. This process is called \emph{Euler-Poincar\'e reduction} and the equations obtained in this way are called the \emph{Euler-Poincar\'e equations} on $ \mathfrak{g}$, see \citet[Section 13.5]{MaRa1999} for details as well as an historical overview. Euler-Poincar\'e reduction is a particular instance of Lagrangian reduction, \cite{MaSc1993a,MaSc1993b} and \cite{CeMaRa2001}, in which one considers a Lagrangian $L:TQ \rightarrow \mathbb{R}  $ invariant under the tangent lifted action of a free and proper group action of a Lie group on $Q$.

For finite dimensional simple thermodynamical systems, it is natural to assume that the group $G$ does not act on the entropy variable $S$. Therefore, the reduced (extended) velocity phase space will be $(TG\times\R)/G\cong\LieAlgebraFont{g}\times\R$, the $\R$ factor being the space in which the entropy variable $S$ of the system lives. In the following, we are going to generalize Euler-Poincar\'e reduction by following the variational formalism for nonequilibrium thermodynamics reviewed in the previous section. Instead of generalizing the \emph{basic} Euler-Poincar\'e equations that we just recalled, we will generalize the Euler-Poincar\'e equations with \emph{advected parameters}, which are very useful in applications. These parameters, initially fixed, acquire dynamics after reduction in the form of an \emph{advection equation}. In this case the Lagrangian is only invariant under the isotropy subgroup of a given reference parameter. The case of advected parameters taking values in (the dual of) a vector space has been studied in \citet[Section 3]{HoMaRa1998} and the general case of advected parameters taking values in manifolds has been developed in \cite{GBTr2010}. We shall follow here this more general setting.

Denoting by $M$ the manifold in which the parameters live we consider a left action of $G$ on $M$, simply denoted by concatenation as $(g,a) \mapsto ga$. The infinitesimal generator associated to $ \xi \in \mathfrak{g}  $ is the vector field on $M$ denoted $ \xi _M$. Given a reference parameter $a_\text{ref}\in M$, we use the notation $G_{a_{\rm ref}} \subset G$ and $\Orb(a_\text{ref}) \subset M$ for the isotropy subgroup and the orbit of $a_{\rm ref}$, respectively. Assuming $G_{a_{\rm ref}}$-invariance, the reduced (extended) velocity phase space is $(TG\times\R)/G_{a_\text{ref}}\cong\LieAlgebraFont{g}\times\Orb(a_\text{ref})\times\R\subset\LieAlgebraFont{g}\times M\times\R$, where the isomorphism is the map $[g,\dot{g},S]\mapsto(g^{-1}\dot{g},g^{-1}a_\text{ref},S)$. In what follows, unless necessary, all actions will be typed with concatenations for the sake of simplicity.

\begin{definition}[\textbf{Reduced map}]\label{def:left-invariant-map}
	Let $G$ be a Lie group acting on the left on a manifold $M$ and let $a_{\rm ref} \in M$ a given element. A map $\Phi_{a_{\rm ref}}:TG\times\R\to\R$ is \emph{left $G_{a_{\rm ref}}$-invariant} if and only if 
	\[
		\Phi_{a_{\rm ref}}(gh,g\dot{h},S)=\Phi_{a_{\rm ref}}(h,\dot{h},S),
	\]
	for all $g \in G_{a_{\rm ref}}$, $(h,\dot{h}) \in  TG$, and $S\in\R$. Left $G_{a_{\rm ref}}$-invariance permits us to define a \emph{reduced map} $\varphi:\LieAlgebraFont{g}\times \operatorname{Orb}(a_{\rm ref}) \times\R\to\R$ by setting
	\[
		\varphi(g^{-1}\dot{g},g^{-1}a_{\rm ref},S)=\Phi_{a_{\rm ref}}(g,\dot{g},S),
	\]
	for all $(g,\dot{g})\in TG$ and $S\in\R$. In the case where $\Phi_{a_{\rm ref}}$ takes values in $T^*G$ instead of $\R$ and is fiber-preserving\footnote{This will be the case of external and friction forces.}, we say that $\Phi_{a_{\rm ref}}$ is \emph{left $G_{a_{\rm ref}}$-equivariant} if and only if
	\[
		\Phi_{a_{\rm ref}}(gh,g\dot{h},S)=g\,\Phi_{a_{\rm ref}}(h,\dot{h},S),
	\]
	for all $g \in G_{a_{\rm ref}}$, $(h,\dot{h})\in TG$ and $S\in\R$. Here $g$ acts on $\Phi(h,\dot{h},a,S)\in T_h^*G$ as the cotangent lift of left translation $L_g$. Left $G_{a_{\rm ref}}$-equivariance permits us to define the \emph{reduced map} $\varphi:\LieAlgebraFont{g}\times \operatorname{Orb}(a_{\rm ref}) \times\R\to\LieAlgebraFont{g}^*$ by setting
	\[
		\varphi(g^{-1}\dot{g},g^{-1}a_{\rm ref},S)=g^{-1}\Phi_{a_{\rm ref}}(g,\dot{g},S),
	\]
	for all $(g,\dot{g})\in TG$ and $S\in\R$.
\end{definition}

We can now state and prove our extended Euler-Poincar\'e reduction theorem for simple thermodynamical systems.

\begin{theorem}[\textbf{Euler-Poincar\'e reduction for simple thermodynamical systems}]\label{thm:TEP-reduction}
	Let $G$ be a Lie group and $\LieAlgebraFont{g}$ its Lie algebra, and let $M$ be a manifold on which $G$ acts on the left. For a fixed parameter $a_\text{ref}\in M$, let:
	\begin{itemize}
		\item $\Lag_{a_\text{ref}}:TG\times\R\to\R$ be a $G_{a_\text{ref}}$-invariant Lagrangian,
		\item $F^\text{ext}_{a_\text{ref}}$, $F^\text{fr}_{a_\text{ref}}:TG\times\R\to T^*G$ be $G_{a_\text{ref}}$-equivariant exterior and friction forces\footnote{As usual, forces are assumed to be fiber-preserving.},
		\item $P^\text{ext}_{H,\,a_\text{ref}}:TG\times\R\to\R$ be a $G_{a_\text{ref}}$-invariant external heat power.
	\end{itemize}
	Denote the corresponding reduced maps by $\lag:\LieAlgebraFont{g}\times\Orb(a_\text{ref})\times\R\to\R$, $f^\text{ext}$, $f^\text{fr}:\LieAlgebraFont{g}\times\Orb(a_\text{ref})\times\R\to\LieAlgebraFont{g}^*$ and $p_H^\text{ext}:\LieAlgebraFont{g}\times\Orb(a_\text{ref})\times\R\to\R$. Then the following assertions are equivalent:
	\begin{enumerate}[label=(\arabic*), ref=\thetheorem.(\arabic*)]
		\item\label{eq:TVP} The curve $(g(t), S(t)) \in G \times \mathbb{R}$ is critical for the variational formulation of nonequilibrium thermodynamics recalled in Section \ref{sec:TVP}, i.e.,
	\[
		\delta\int_0^T\Lag_{a_\text{ref}}(g,\dot{g},S)\,\mathrm{d}t+\int_0^T\big\langle F^\text{ext}_{a_\text{ref}}(g,\dot{g},S),\delta g\big\rangle\,\mathrm{d}t=0
	\]
subject to the variational and phenomenological constraints
\begin{align*} 
	\frac{\partial\Lag_{a_\text{ref}}}{\partial S}(g,\dot{g},S)\delta S&=\left\langle F^\text{fr}_{a_\text{ref}}(g,\dot{g},S),\delta g\right\rangle,\\
	\frac{\partial\Lag_{a_\text{ref}}}{\partial S}(g,\dot{g},S)\dot{S}&=\left\langle F^\text{fr}_{a_\text{ref}}(g,\dot{g},S),\dot{g}\right\rangle-P^\text{ext}_{H,\,{a_\text{ref}}}(g,\dot{g},S),
\end{align*}
where $ \delta g$ vanishes at $t=0, T$.
\item The curve $(g(t),S(t)) \in G \times \mathbb{R}$ satisfies the equations of motion for the simple thermodynamical system, i.e., equations \eqref{eq:mechanical-equation} and \eqref{eq:thermodynamical-equation}.
		
		\item\label{eq:RTVP} The curve $( \xi (t), a(t), S(t))\in \mathfrak{g}  \times \operatorname{Orb}(a_{\rm ref}) \times \mathbb{R}  $, defined by $ \xi (t)= g(t) ^{-1} \dot g(t)$, $a(t)= g(t)^{-1} a_{\rm ref}$, is critical for the reduced variational formulation
\[
\delta\int_0^T\lag(\xi,a,S)\,\mathrm{d}t+\int_0^T\left\langle f^\text{ext}(\xi,a,S),\eta\right\rangle\mathrm{d}t=0
\]
subject to the variational and phenomenological  constraints
\begin{align*} 
			\frac{\partial\lag}{\partial S}(\xi,a,S)\delta S&=\left\langle f^\text{fr}(\xi,a,S),\eta\right\rangle,\\[2mm]
			\frac{\partial\lag}{\partial S}(\xi,a,S)\dot{S}&=\left\langle f^\text{fr}(\xi,a,S\big),\xi\right\rangle-p^\text{ext}_H(\xi,a,S),
\end{align*}
and to the Euler-Poincar\'e constraints
		\[
			\delta\xi=\dot{\eta}+[\xi,\eta],\quad\delta a=-\eta_M(a),
		\]
		where $\eta$ is any curve in $\LieAlgebraFont{g}$ vanishing at $t=0,T$.		
		\item\label{eq:TEP-equations} The curve $( \xi (t), a(t), S(t))\in \mathfrak{g}  \times \operatorname{Orb}(a_{\rm ref}) \times \mathbb{R}  $, defined by $ \xi (t)= g(t) ^{-1} \dot g(t)$, $a(t)= g(t)^{-1} a_{\rm ref}$ satisfies the equations
		\begin{empheq}[left=\empheqlbrace]{align}
			& \frac{\mathrm{d}}{\mathrm{d}t}\frac{\partial\lag}{\partial\xi}(\xi,a,S)=\ad^*_\xi\frac{\partial\lag}{\partial\xi}(\xi,a,S)-\mathbf{J}\left(\frac{\partial\lag}{\partial a}(\xi,a,S)\right)+f^\text{ext}(\xi,a,S)+f^\text{fr}(\xi,a,S),\label{eq:TEP-equation}\\[2mm]
			& \frac{\partial\lag}{\partial S}(\xi,a,S)\dot{S}=\left\langle f^\text{fr}(\xi,a,S),\xi\right\rangle-p_H^\text{ext}(\xi,a,S),\\[2mm]
			& \dot{a}+\xi_M(a)=0\label{eq:advection-equation},
		\end{empheq}
		where $\mathbf{J}:T^*M\to\LieAlgebraFont{g}^*$ is the momentum map associated to the cotangent lift of the action of $G$ on $T^*M$; it is defined by $\big\langle\mathbf{J}(\alpha_x),\xi\big\rangle=\big\langle\alpha_x,\xi_M(x)\big\rangle$ for any $x\in M$, $\alpha_x\in T_x^*M$ and $\xi\in\LieAlgebraFont{g}$.
	\end{enumerate}
\end{theorem}

\begin{proof} The equivalence between (1) and (2) follows from a direct computation, see \citet[Section 2]{GBYo2017a}. We now show that (3) and (4) are equivalent. Taking variations from the left hand side of the variational condition in (3), we obtain using all the available constraints:
	\begin{align*}
			&\int_0^T\left\langle\frac{\partial\lag}{\partial\xi},\delta\xi\right\rangle\,\mathrm{d}t
			+\int_0^T\left\langle\frac{\partial\lag}{\partial a},\delta a\right\rangle\,\mathrm{d}t
			+\int_0^T\frac{\partial\lag}{\partial S}\delta S\,\mathrm{d}t
			+\int_0^T\langle f^\text{ext},\eta\rangle\,\mathrm{d}t\\
			&\quad=\int_0^T\left\langle\frac{\partial\lag}{\partial\xi},\dot{\eta}+[\xi,\eta]\right\rangle\,\mathrm{d}t
			-\int_0^T\left\langle\frac{\partial\lag}{\partial a},\eta_M(a)\right\rangle\,\mathrm{d}t
			+\int_0^T\langle f^\text{ext}+f^\text{fr},\eta\rangle\,\mathrm{d}t\\
			&\quad=\int_0^T\left\langle-\frac{\mathrm{d}}{\mathrm{d}t}\frac{\partial\lag}{\partial\xi}+\ad^*_\xi\frac{\partial\lag}{\partial\xi}-\mathbf{J}\left(\frac{\partial\lag}{\partial a}\right)+f^\text{ext}+f^\text{fr},\eta\right\rangle\,\mathrm{d}t+\left[\left\langle\frac{\partial\lag}{\partial\xi},\eta\right\rangle\right]_0^T,
	\end{align*}
	which yields equations \eqref{eq:TEP-equation}, the last term in the right hand side being zero. The advection equation comes from a slightly more technical computation. Denoting by $\sigma:G\times M\to M$ the left action of $G$ on $M$ and by $ \mathbf{d} \sigma_{(g,a)}:T_gG \times T_a \operatorname{Orb}(a_{\rm ref}) \rightarrow \mathbb{R}  $ its derivative at $(g,a)$, we compute that
	\begin{equation*}
		\dot{a}=\dif\sigma_{(g^{-1},a_\text{ref})}(-g^{-1}\dot{g}g^{-1},0).
	\end{equation*}
	Introducing for $k\in G$ the map $E_k:G\times M\to G\times M$ defined by $E_k(h,n)=(hk^{-1},kn)$ we obtain
	\[
		\dot{a}=E_{g}^*\dif\sigma_{(e,a)}(-\xi,0).
	\]
	Since the pullback commutes with the differential and $E_k^*\sigma=\sigma$ for any $k\in G$, we get the advection equation $\dot{a}=\mathbf{d} \sigma _{(e,a)}(- \xi , 0)=-\xi_M(a)$.

	{The equivalence between (1) and (3) follows by observing that the constraints and action functional in (1) and (3) are equal from the $G_{a_\text{ref}}$-invariance of the Lagrangian and external heat power, as well as the $G_{a_\text{ref}}$-equivariance of the forces. For instance, we have
	\[
		\big\langle F_{a_\text{ref}}^\text{ext}(g,\dot{g},S),\delta g\big\rangle
		=\big\langle g\cdot f^\text{ext}(\xi,a,S),\delta g\big\rangle
		=\big\langle f^\text{ext}(\xi,a,S),\eta\big\rangle.
	\]
	The equivalence between the variations used in (1) and (3) follows exactly as in the case without thermodynamics.}
	\end{proof}

\begin{remark}
	The theorem we just proved could be called the \emph{left-left} Euler-Poincar\'e reduction for simple thermodynamical systems because we used the left action of $G$ on itself, as well as a left action of $G$ on $M$, but other combinations are possible, and are useful for applications. For the example presented in this article (see Section \ref{sec:ball}), we will stick to this left-left version of the theorem. In absence of thermodynamics and external forces, the system \eqref{eq:TEP-equation}--\eqref{eq:advection-equation} reduces to the Euler-Poincar\'e equations
\begin{empheq}[left=\empheqlbrace]{align}
			& \frac{\mathrm{d}}{\mathrm{d}t}\frac{\partial\lag}{\partial\xi}(\xi,a,S)=\ad^*_\xi\frac{\partial\lag}{\partial\xi}(\xi,a,S)-\mathbf{J}\left(\frac{\partial\lag}{\partial a}(\xi,a,S)\right),\nonumber\\[2mm]
			& \dot{a}+\xi_M(a)=0\nonumber,
		\end{empheq}	
with advected parameter in the manifold $M$, see \cite{HoMaRa1998}, \cite{GBTr2010}.
\end{remark}

\begin{remark}[Coadjoint orbits] {We note that in general, the solution of \eqref{eq:TEP-equation}--\eqref{eq:advection-equation} do not preserve the coadjoint orbits in the dual of the semidirect product Lie algebra $(\mathfrak{g}\,\circledS\, V)^*$, which are well-known to be preserved in absence of friction forces and external effects, when $M=V^*$ is the dual of a vector space on which $G$ acts by a representation, see \cite{HoMaRa1998}. It is however possible to choose the friction force in such a way that the coadjoint orbits are preserved.}

{To simplify our discussion, let us assume that there are no advected parameters, so that we have $G$-invariance, and assume that external effects are absent, $f^{\rm ext}=0$, $p^{\rm ext}_H=0$. We assume that the Lagrangian is hyperregular and consider the associated Hamiltonian $h:\mathfrak{g}^*\times\mathbb{R}\rightarrow\mathbb{R}$ defined via the Legendre transform as $h(\mu,S)=\langle \mu,\xi\rangle-\ell(\xi,S)$, where $\xi$ is such that $\frac{\partial\ell}{\partial\xi}(\xi,S)=\mu$. In this case, the thermodynamical system \eqref{eq:TEP-equation}--\eqref{eq:advection-equation} reduces to
\begin{empheq}[left=\empheqlbrace]{align}
			& \displaystyle\frac{\mathrm{d}}{\mathrm{d}t}\mu=\ad^*_{\frac{\partial h}{\partial\mu}(\mu,S)}\mu +f^\text{fr}(\mu,S),\label{eq:TEP-equation_EP}\\[2mm]
			& \frac{\partial h}{\partial S}(\mu,S)\dot{S}=-\left\langle f^\text{fr}(\mu,S),\frac{\partial h}{\partial\mu}(\mu,S)\right\rangle,
		\end{empheq}
where we have expressed the friction force in terms of the momentum $\mu$. Recall that the tangent space at $\mu$ to a coadjoint $\mathcal{O}_{\mu_0}=\{\Ad_g^*\mu_0\mid g\in G\}\subset \mathfrak{g}^*$ is $T_\mu \mathcal{O}_{\mu_0}=\{\ad^*_\xi\mu\mid \xi\in \mathfrak{g}\}$, see, e.g., \cite{MaRa1999}. From this expression of the tangent space and from equation \eqref{eq:TEP-equation_EP} it is clear that the coadjoint orbits are preserved if and only if the friction force is of the form $f^\text{fr}(\mu,S)=\ad^*_{\zeta(\mu,S)}\mu$, for a function $\zeta:\mathfrak{g}^*\times\mathbb{R}\rightarrow \mathfrak{g}$. In this case, we have $(\mu(t),S(t))\in \mathcal{O}_{\mu_0}\times \mathbb{R}$ for all $t\geq 0$, where $\mu_0=\mu(0)$ is the initial condition for the momentum. From the second law and equation \eqref{eq:TEP-equation_EP}, the friction force must be dissipative. Since $\left\langle f^\text{fr}(\mu,S),\frac{\partial h}{\partial\mu}(\mu,S)\right\rangle=- \left\langle \ad^*_{\frac{\partial h}{\partial\mu}(\mu,S)}\mu,\zeta(\mu,S)\right\rangle $, the choice $\zeta(\mu,S):=\lambda(\mu,S) \left[\ad^*_{\frac{\partial h}{\partial\mu}(\mu,S)}\mu\right]^\sharp$, for a positive function  $\lambda:\mathfrak{g}^*\times\mathbb{R}\rightarrow\mathbb{R}$, yields a dissipative force, where $\sharp:\mathfrak{g}^*\rightarrow\mathfrak{g}$ is the sharp operator associated to an inner product $\gamma:\mathfrak{g}\times\mathfrak{g}\rightarrow\mathbb{R}$ on $\mathfrak{g}$, i.e., we have
\begin{equation}\label{fr_db}
f^\text{fr}(\mu,S)=\lambda(\mu,S)\ad^*_{\big(\ad^*_{\frac{\partial h}{\partial\mu}(\mu,S)}\mu\big)^\sharp}\mu.
\end{equation}
Note that the dependence of $\lambda$ on the entropy can be converted to a dependence on the temperature of the system.
In this case, the entropy equation reduces to
\[
\displaystyle T\dot S= \lambda(\mu,S)\Big|\ad^*_{\frac{\partial h}{\partial\mu}}(\mu,S)\Big|^2,
\]
where the norm $|\cdot |$ is associated to $\gamma$. In absence of the entropy variable, \eqref{fr_db} recovers the expression of the dissipative external force obtained by double bracket dissipation in \cite{BlKrMaRa1994}. In our context, $f^\text{fr}(\mu,S)$, as given in \eqref{fr_db}, is an internal force describing an irreversible process occurring in the system, and leading to an increase of the entropy.}
\end{remark}

\subsection{Energy balance}

We are now interested in writing the reduced version of the energy balance \eqref{eq:energy-balance}. We will work with the notations introduced in Theorem \ref{thm:TEP-reduction}. From the definition \eqref{def:energy} of the energy, we define its reduced version as the map $e:\LieAlgebraFont{g}\times\Orb(a_\text{ref})\times\R\to\R$ given by
\begin{equation}\label{def:reduced-energy}
	e(\xi,a,S)=\left\langle\frac{\partial\lag}{\partial\xi}(\xi,a,S),\xi\right\rangle-\lag(\xi,a,S),
\end{equation}
for all $\xi\in\LieAlgebraFont{g}$, $a\in M$ and $S\in\R$.

\begin{corollary}[\textbf{Energy balance}]\label{cor:energy-balance}
	Suppose that $(g,\xi,a,S)$ is a curve that satisfies the equations \ref{eq:TEP-equations}. Let $p^\text{ext}_W:\mathfrak{g}\times\Orb(a_\text{ref})\times\R\to\R$ be the reduced power of the external forces, defined by $p^\text{ext}_W(\xi,a,S)=\big\langle f^\text{ext}(\xi,a,S),\xi\big\rangle$. Then along the curve $(g,\xi,a,S)$ we have
	\[
		\frac{\mathrm{d}e}{\mathrm{d}t}=p^\text{ext}_W(\xi,a,S)+p^\text{ext}_H(\xi,a,S).
	\]
\end{corollary}

\begin{proof}
	Using the equations \ref{eq:TEP-equations} we have:
	\begin{align*}
		\frac{\mathrm{d}e}{\mathrm{d}t}&=
		\left\langle\frac{\mathrm{d}}{\mathrm{d}t}\frac{\partial\lag}{\partial\xi}(\xi,a,S)+\mathbf{J}\left(\frac{\partial\lag}{\partial a}(\xi,a,S)\right),\xi\right\rangle-\frac{\partial\lag}{\partial S}(\xi,a,S)\dot{S}\\
		&=\left\langle\ad_\xi^*\frac{\partial\lag}{\partial\xi}(\xi,a,S),\xi\right\rangle+\big\langle f^\text{ext}(\xi,a,S),\xi\big\rangle+p^\text{ext}_H(\xi,a,S)\\
		&=p^\text{ext}_W(\xi,a,S)+p^\text{ext}_H(\xi,a,S).
	\end{align*}	
\end{proof}

\subsection{Kelvin-Noether theorem}\label{sec:Kelvin-Noether}

The Kelvin-Noether theorem is a version of Noether's theorem that holds for solutions of the Euler-Poincar\'e equations, see \citet[Theorem 4.1]{HoMaRa1998}. It is especially useful to understand the Kelvin circulation theorem in fluid dynamics. We shall give here an extension of this theorem which includes thermodynamics, in the finite dimensional case.

Let $\mathcal{C}$ be a manifold on which $G$ acts on the left and let $\mathcal{K}:\mathcal{C}\times \operatorname{Orb}(a_{\rm ref}) \to\LieAlgebraFont{g}^{**}\cong\LieAlgebraFont{g}$ be a $G$-equivariant map, where the action on $\LieAlgebraFont{g}^{**}\cong\LieAlgebraFont{g}$ is the dual of the coadjoint action of $G$ on $\LieAlgebraFont{g}^*$ (we do identify $\LieAlgebraFont{g}^{**}$ with $\LieAlgebraFont{g}$ because our examples will be finite-dimensional). The \emph{Kelvin-Noether quantity} associated to $ \mathcal{C} $ and $ \mathcal{K} $ is the map $I:\mathcal{C}\times\LieAlgebraFont{g}\times  \operatorname{Orb}(a_{\rm ref}) \times\R\to\R$ defined by
\[
	I(c,\xi,a,S)=\left\langle\mathcal{K}(c,a),\frac{\partial\lag}{\partial\xi}(\xi,a,S)\right\rangle,
\]
for all $c\in\mathcal{C}$, $\xi\in\LieAlgebraFont{g}$, $a\in  \operatorname{Orb}(a_{\rm ref}) $ and $S\in\R$.

\begin{corollary}[\textbf{Kelvin-Noether theorem}]\label{cor:Kelvin-Noether}
	Let $c_\text{ref}\in\mathcal{C}$ fixed and let $(g,\xi,a,S)$ be a curve satisfying the reduced equations \eqref{eq:TEP-equation}--\eqref{eq:advection-equation}. Let $c=g^{-1}c_\text{ref}$. Then along the curve $(g,\xi,a,S)$ we have:
	\[
		\frac{\mathrm{d}I}{\mathrm{d}t}=\left\langle\mathcal{K}(c,a),-\mathbf{J}\left(\frac{\partial\lag}{\partial a}(\xi,a,S)\right)+\big(f^\text{ext}+f^\text{fr}\big)(\xi,a,S)\right\rangle.
	\]
\end{corollary}

\begin{proof}
	Write $a=g^{-1}a_\text{ref}$, with $a_\text{ref}=g(0)a(0)$. First, using the $G$-equivariance property of $\mathcal{K}$, we have that
	\[
		\left\langle\mathcal{K}(c,a),\frac{\partial\lag}{\partial\xi}(\xi,a,S)\right\rangle
		=\left\langle\mathcal{K}(c_\text{ref},a_\text{ref}), \operatorname{Ad}^*_{ g ^{-1} } \frac{\partial\lag}{\partial\xi}(\xi,a,S)\right\rangle.
	\]
	Then, using the formula for the differentiation of the coadjoint action as well as equations \eqref{eq:TEP-equation}, we obtain:
	\begin{align*}
		\frac{\mathrm{d}I}{\mathrm{d}t}
		&=\left\langle\mathcal{K}(c_\text{ref},a_\text{ref}),  \operatorname{Ad}^*_{ g ^{-1} }\left[-\ad^*_{\xi}\frac{\partial\lag}{\partial\xi}(\xi,a,S)+\frac{\mathrm{d}}{\mathrm{d}t}\frac{\partial\lag}{\partial\xi}(\xi,a,S)\right]\right\rangle\\
		&=\left\langle\mathcal{K}(c_\text{ref},a_\text{ref}),   \operatorname{Ad}^*_{ g ^{-1} }\left[-\mathbf{J}\left(\frac{\partial\lag}{\partial a}(\xi,a,S)\right)+\big(f^\text{ext}+f^\text{fr}\big)(\xi,a,S)\right]\right\rangle\\
		&=\left\langle\mathcal{K}(c,a),-\mathbf{J}\left(\frac{\partial\lag}{\partial a}(\xi,a,S)\right)+\big(f^\text{ext}+f^\text{fr}\big)(\xi,a,S)\right\rangle.
	\end{align*}
\end{proof}

\section{Variational discretization of simple thermodynamical Euler-Poincar\'e systems}

In this section we first review from \cite{GBYo2017c} the variational discretization for the thermodynamics of simple systems. Then we develop the discrete version of the Euler-Poincar\'e reduction for thermodynamics carried out in Section \ref{sec:TEP-reduction}.
  
\subsection{Variational discretization of thermodynamics}\label{sec:variational-integrators}

Variational integrators are numerical schemes that arise from a discrete version of Hamilton's principle, or Lagrange-d'Alembert's principle in the case external forces act on the system. These geometric integrators are thoroughly reviewed in \cite{MaWe2001}, we simply recall the broad idea hereafter. Let $Q$ be a configuration manifold and $\Lag:TQ\to\R$ be a Lagrangian. Given a time step $h$, $[0,T]$ is discretized into the sequence $t_k=kh$, $k\in\{0,\dots,N\}$. A curve $q$ in $Q$ is discretized into a sequence $q_d=(q_k)_{0\leq k\leq N}$, and a variation $\delta q$ of $q$ is discretized into a sequence $\delta q_d=(\delta q_k)_{0\leq k\leq N}$, such that $\delta q_k\in T_{g_k}Q$, for any $k\in\{0,\dots,N\}$. The Lagrangian $\Lag$ is discretized into a \emph{discrete Lagrangian} $\Lag_d:Q\times Q\to\R$ such that we have
\[
	\Lag_d(q_k,q_{k+1})\approx\int_{t_k}^{t_{k+1}}\Lag(q(t),\dot{q}(t))\,\mathrm{d}t,
\]
where the curve $q(t)$ is the solution of the Euler-Lagrange equations with endpoints $q_k$ and $q_{k+1}$. Usually this approximation is related to some numerical quadrature rule of the integral above. Then the discrete analogue of Hamilton's principle for the discrete action defined by
\[
	\mathcal{S}_d(q_d)=\sum_{k=0}^{N-1}\Lag_d(q_k,q_{k+1})
\]
is $\delta\mathcal{S}_d(q_d)\cdot\delta q_d=0$ for all variations $\delta q_d$ of $q_d$ with vanishing endpoints. After taking variations and applying a discrete integration by parts formula (change of indices), we obtain the \emph{discrete Euler-Lagrange} equations:
\[
	D_2\Lag_d(q_{k-1},q_k)+D_1\Lag_d(q_k,q_{k+1})=0,\quad\forall k\in\{1,\dots,N-1\}.
\]
For more details see \citet[Section 1.3.1]{MaWe2001}. These equations define, under appropriate conditions, an algorithm which solves for $q_{k+1}$ knowing the two previous configuration variables $q_k$ and $q_{k-1}$. The fact that these integrators are symplectic will be reviewed later on.

We now review the variational discretization of thermodynamical simple systems as introduced in \citet[Section 3.1]{GBYo2017c}. The entropy curve $t\in[0,T]\mapsto S(t)\in\R$ is discretized into a sequence $S_d=(S_k)_{0\leq k\leq N}$, $k\in\{0,\dots,N\}$. The discrete Lagrangian is now a map $\Lag_d:Q\times Q\times\R\times\R\to\R$ such that:
\[
	\Lag_d(q_k,q_{k+1},S_k,S_{k+1})\approx\int_{t_k}^{t_{k+1}}\Lag(q(t),\dot{q}(t),S(t))\,\mathrm{d}t.
\]
As recalled in Section \ref{sec:TVP}, we have two kind of forces that act on the system: external forces $F^\text{ext}$ (that do not derive from a potential) and friction forces $F^\text{fr}$. The discrete counterparts of these forces are given by four maps $F^{\text{ext}+}_d$, $F^{\text{ext}-}_d$, $F^{\text{fr}+}_d$, $F^{\text{fr}-}_d:Q\times Q\times\R\times\R\to T^*Q$ such that the following approximation holds:
\begin{align*}
	\int_{t_k}^{t_{k+1}}\left\langle F^{\text{ext}}(q(t),\dot{q}(t),S(t)),\delta q(t)\right\rangle\mathrm{d}t
	\approx\big\langle & F_{d}^{\text{ext}-}(q_k,q_{k+1},S_k,S_{k+1}),\delta q_k\big\rangle\\
	&\quad+\big\langle F_{d}^{\text{ext}+}(q_k,q_{k+1},S_k,S_{k+1}),\delta q_{k+1}\big\rangle
\end{align*}
and similarly for $F^\text{fr}_d$. These discrete forces are required to be fiber-preserving in the sense that $\pi_{T^*Q}\circ F^{\text{ext}\pm}_d=\pi_Q^\pm$ and similarly for $F^\text{fr}_d$, with $\pi_{T^*Q}:T^*Q\to Q$ being the canonical projection and the maps $\pi_Q^\pm:Q\times Q\times\R\times\R\to Q$ being defined by $\pi_Q^-(q_0,q_1,S_0,S_1)=q_0$ and $\pi_Q^+(q_0,q_1,S_0,S_1)=q_1$. Concretely this means, for instance, that $F^{\text{ext}+}_d(q_0,q_1,S_0,S_1)\in T_{q_1}^*Q$. See also \citet[Section 3.2.1]{MaWe2001} for a description of forces in the discrete setting.

We now need to discretize the phenomenological constraint. As stated in \citet[Section 3]{GBYo2017c}, this is done with the help of a \emph{finite difference map} $\varphi:Q\times Q\times\R\times\R\to TQ\times T\R$, a notion which was introduced in \citet[Section 4]{McPe2006} for the development of variational integrators for systems with nonholonomic constraints. Essentially, such maps are directly responsible for the discretization of $q(t)$, $\dot{q}(t)$, $S(t)$ and $\dot{S}(t)$ in terms of $q_k$, $q_{k+1}$, $S_k$ and $S_{k+1}$, therefore their use is not limited to the discretization of constraints and they are also used for the discretization of the Lagrangian, as we will see later.

The phenomenological constraint can be seen as the zero-level set $\mathcal{C}$ of the map $P:T(Q\times\R)\to\R$ defined by:
\[
	P(q,\dot{q},S,\dot{S})=\frac{\partial\Lag}{\partial S}(q,\dot{q},S)\dot{S}-\big\langle F^\text{fr}(q,\dot{q},S),\dot{q}\big\rangle+P_H^\text{ext}(q,\dot{q},S),
\]
for any $(q,\dot{q})\in TQ$ and $(S,\dot{S})\in T\R$. Note that since $\frac{\partial P}{\partial\dot{S}}=-T\neq 0$, $P$ is a submersion and $\mathcal{C}$ is a codimension one submanifold of $T(Q\times\R)$. The discrete counterpart of $\mathcal{C}$ is $\mathcal{C}_d$, defined as
\begin{equation}\label{eq:discrete-phenomenological-constraint}
	\mathcal{C}_d=\varphi^{-1}(\mathcal{C})\subset Q\times Q\times\R\times\R.
\end{equation}
Therefore, $\mathcal{C}_d$ can be seen as the zero-level set of the map $P_d=P\circ\varphi:Q\times Q\times\R\times\R\to\R$. 

\begin{remark}
	Note that the way in which the entropy is discretized in the discrete Lagrangian $\Lag_d$ is not necessarily the one used for the discrete phenomenological constraint $P_d$.
\end{remark}

We can now state the discrete version of the variational formulation of Section \ref{sec:TVP}.
A discrete curve $(q_d,S_d)$ satisfies the \textit{discrete variational formulation for nonequilibrium thermodynamics} if first it satisfies the \emph{discrete variational condition}
\begin{align}
	\delta\sum_{k=0}^{N-1}\Lag_d(q_k&,q_{k+1},S_k,S_{k+1})\notag\\
	&+\sum_{k=0}^{N-1}\big\langle F^{\text{ext}-}_d(q_k,q_{k+1},S_k,S_{k+1}),\delta q_k\big\rangle
	+\big\langle F^{\text{ext}+}_d(q_k,q_{k+1},S_k,S_{k+1}),\delta q_{k+1}\big\rangle=0,
\end{align}
for all variations $\delta q_d$ and $\delta S_d$ satisfying the \emph{discrete variational constraint}
\begin{align}
	D_3\Lag_d(q_k,q_{k+1},S_k,S_{k+1})&\delta S_k+D_4\Lag_d(q_k,q_{k+1},S_k,S_{k+1})\delta S_{k+1}\notag\\
	&=\big\langle F^{\text{fr}-}_d(q_k,q_{k+1},S_k,S_{k+1}),\delta q_k\big\rangle
	+\big\langle F^{\text{fr}+}_d(q_k,q_{k+1},S_k,S_{k+1}),\delta q_{k+1}\big\rangle,
\end{align}
for all $k\in\{0,\dots,N-1\}$, where $ \delta q_d$ vanishes at the endpoints, and if it also satisfies the \emph{discrete phenomenological constraint} 
\begin{equation}
	P_d(q_k,q_{k+1},S_k,S_{k+1})=0,
\end{equation}
for all $k\in\{0,\dots,N-1\}$. Taking variations and applying a discrete integration by parts (change of indices) yield the \emph{discrete equations for the thermodynamic of simple closed systems}:
\begin{empheq}[left=\empheqlbrace,right=,]{align}
	&D_1\Lag_d(q_k,q_{k+1},S_k,S_{k+1})+D_2\Lag_d(q_{k-1},q_k,S_{k-1},S_k)\notag\\[2mm]
	&\quad\quad\quad+(F^{\text{ext}-}_d+F^{\text{fr}-}_d)(q_k,q_{k+1},S_k,S_{k+1})+(F^{\text{ext}+}_d+F^{\text{fr}+}_d)(q_{k-1},q_k,S_{k-1},S_k)=0,\label{eq:discrete-mechanical-equation}\\[2mm]
	&P_d(q_k,q_{k+1},S_k,S_{k+1})=0\label{eq:discrete-thermodynamical-equation},
\end{empheq}
for all $k\in\{0,\dots,N-1\}$, see \cite{GBYo2017c}.

\subsection{Discrete Euler-Poincar\'e reduction for simple thermodynamical systems}\label{sec:DTEP-reduction}

Starting from the discrete variational formalism reviewed in the previous section in the case where $Q$ is a finite-dimensional Lie group $G$ acting on itself via  left multiplication, we perform a discrete analogue of the Euler-Poincar\'e reduction for simple thermodynamical systems developed in Section \ref{sec:TEP-reduction}. The tangent space $TG$ is discretized into $G\times G$ as usual, and the analogue of the projection map $\bar{\pi}:TG\to TG/G\cong\LieAlgebraFont{g}$ is given by one of the two maps $\bar{\pi}^\pm:G\times G\to(G\times G)/G\cong G$ defined by $\bar{\pi}^+(g_0,g_1)=g_0^{-1}g_1$ and $\bar{\pi}^-(g_0,g_1)=g_1^{-1}g_0$. In what follows we choose to only work with $\bar{\pi}^+$ and will write $\Xi_k$ for $g_k^{-1}g_{k+1}$; this is the discrete analogue of $TG/G$ being identified with $\LieAlgebraFont{g}$ with the help of the left Maurer-Cartan form. Note however that this is just a matter of choice. In the case where the Lagrangian is $G$-invariant with respect to the right multiplication, we would define the maps $\bar{\pi}^\pm$ by $\bar{\pi}^+(g_0,g_2)=g_1g_0^{-1}$ and $\bar{\pi}^-(g_0,g_1)=g_0g_1^{-1}$. Also note that contrary to the continuous Euler-Poincar\'e reduction, the discrete reduced tangent space is represented by the manifold $G$ rather than the vector space $\LieAlgebraFont{g}$.

\begin{definition}[\textbf{Discrete reduced map}]
	Let $G$ be a Lie group acting on the left of itself as well as on a manifold $M$. A map $\Phi_{d, a_{\rm ref}}:G\times G\times\R\times\R\to\R$ is \emph{left $G_{a_{\rm ref}}$-invariant} if and only if
	\[
		\Phi_{d, a_{\rm ref}}(gh_0,gh_1,S_0,S_1)=\Phi_{d, a_{\rm ref}}(h_0,h_1,S_0,S_1),
	\]
	for all $g \in G_{a_{\rm ref}}$, $h_0, h_1\in G$, and $S_0, S_1\in\R$. Left $G_{a_{\rm ref}}$-invariance permits us to define a \emph{reduced map} $\varphi_d:G\times\operatorname{Orb}(a_{a_{\rm ref}}) \times\R\times\R\to\R$ by setting
	\[
		\varphi_d(g_0^{-1}g_1,g^{-1}a_{\rm ref},S_0,S_1)=\Phi_{d, a_{\rm ref}}(g_0,g_1,S_0,S_1),
	\]
	for all $g_0, g_1\in G$ and $S_0, S_1\in\R$.
\end{definition}

\begin{definition}[\textbf{Discrete reduced forces}]\label{def:DRF} 
	Let $F^\pm_d:G\times G \times\R\times\R\to T^*G$ denote a pair of discrete external forces as explained in Section \ref{sec:variational-integrators}. This pair is \emph{left $G_{a_{\rm ref}}$-equivariant} if and only if
	\[
		F^\pm_{d,a_\text{ref}}(gh_0,gh_1,S_0,S_1)=g\,F_{d,a_\text{ref}}^\pm(h_0,h_1,S_0,S_1),
	\]
	for all $g \in G_{a_{\rm ref}}$, $h_0, h_1\in G$, and $S_0, S_1\in\R$. Note that $g$ acts on $F_d(h_0,h_1,a,S_0,S_1)\in T_h^*G$ as the cotangent lift of $L_g$. Then we define two \emph{reduced discrete forces} $\mathsf{F}^\pm_d:G\times\operatorname{Orb}(a_{\rm ref}) \times\R\times\R\to\LieAlgebraFont{g}^*$ by setting
	\begin{align*} 
		\mathsf{F}^+_{d,a_\text{ref}}(g_0^{-1}g_1,g_0^{-1}a_{\rm ref},S_0,S_1)&=g_1^{-1}F^+_{d,a_\text{ref}}(g_0,g_1,S_0,S_1)\\
		\mathsf{F}^-_{d,a_\text{ref}}(g_0^{-1}g_1,g_0^{-1}a_{\rm ref},S_0,S_1)&=g_0^{-1}F^-_{d,a_\text{ref}}(g_0,g_1,S_0,S_1),
	\end{align*} 
	for all $g_0, g_1\in G$ and $S_0, S_1\in\R$.
\end{definition}

The following theorem extends the discrete Euler-Poincar\'e reduction developed in \citet{MaPeSh1999}, \cite{BoSu1999} to include thermodynamics.


\begin{theorem}[\textbf{Discrete Euler-Poincar\'e reduction for simple thermodynamical systems}]\label{thm:DTEP-reduction}
	Let $G$ be a Lie group and let $\LieAlgebraFont{g}$ be its Lie algebra. Suppose that $G$ acts on the left on a manifold $M$. For a fixed parameter $a_\text{ref}\in M$, let:
	\begin{itemize}
		\item $\Lag_{d,a_\text{ref}}:G\times G\times\R\times\R\to\R$ be a discrete $G_{a_\text{ref}}$-invariant Lagrangian,
		\item $F_{d,a_\text{ref}}^{\text{ext}\pm}$, $F_{d,a_\text{ref}}^{\text{fr}\pm}:G\times G\times\R\times\R\to T^*G$ be discrete $G_{a_\text{ref}}$-equivariant external and friction forces,
		\item $P_{d,a_\text{ref}}:G\times G\times\R\times\R\to\R$ the discrete $G_{a_\text{ref}}$-invariant phenomenological constraint.
	\end{itemize}
	Let $\mathsf{\Lag_d}:G\times\Orb(a_\text{ref})\times\R\times\R\to\R$, $\mathsf{F}^{\text{ext}\pm}_d$, $\mathsf{F}_d^{\text{fr}\pm}:G\times\Orb(a_\text{ref})\times\R\times\R\to\LieAlgebraFont{g}^*$ and $\mathsf{P}_d:G\times\Orb(a_\text{ref})\times\R\times\R\to\R$ be the associated reduced maps, given by Definition \ref{def:DRF}. Then the following assertions are equivalent:
 
	\begin{enumerate}[label=(\arabic*), ref=\thetheorem.(\arabic*)]
		\item\label{eq:DTVP} The discrete curve $(g_d,S_d)$ is critical for the discrete variational formulation for nonequilibrium thermodynamics, i.e., 
\begin{align*} 
&\delta\sum_{k=0}^{N-1}\Lag_{d,a_\text{ref}}(g_k,g_{k+1},S_k,S_{k+1})\\
&\quad\quad\quad+\sum_{k=0}^{N-1}\big\langle F_{a_\text{ref}}^{\text{ext}-}(g_k,g_{k+1},S_k,S_{k+1}),\delta g_k\big\rangle
			+\big\langle F_{a_\text{ref}}^{\text{ext}+}(g_k,g_{k+1},S_k,S_{k+1}),\delta g_{k+1}\big\rangle=0,
\end{align*}
subject to the discrete variational and phenomenological constraint
\begin{align*}
&D_3\Lag_{d,a_\text{ref}}(g_k,g_{k+1},S_k,S_{k+1})\delta S_k+D_4\Lag_{d,a_\text{ref}}(g_k,g_{k+1},S_k,S_{k+1})\delta S_{k+1}\\
& \quad \quad\quad\quad=\big\langle F_{a_\text{ref}}^{\text{fr}-}(g_k,g_{k+1},S_k,S_{k+1}),\delta g_k\big\rangle
+\big\langle F_{a_\text{ref}}^{\text{fr}+}(g_k,g_{k+1},S_k,S_{k+1}),\delta g_{k+1}\big\rangle,\\
&P_{d,a_\text{ref}}(g_k,g_{k+1},S_k,S_{k+1})=0,
\end{align*} 
where $\delta g_d$ vanishes at endpoints.
		\item The discrete curve $(g_d,S_d)$ on $G \times \mathbb{R}$ satisfies the equations \eqref{eq:discrete-mechanical-equation} and \eqref{eq:discrete-thermodynamical-equation} for the thermodynamic of simple closed systems.
		
		\item\label{eq:DRTVP} The discrete curve $(\Xi_d,a_d, S_d)$ on $G \times \operatorname{Orb}(a_{\rm ref}) \times \mathbb{R}  $ defined by $\Xi_k=g_k^{-1}g_{k+1}\in G$ and $a_k=g_k^{-1}a_\text{ref} \in \operatorname{Orb}(a_{\rm ref})$, is critical for the \emph{reduced discrete variational formulation of nonequilibrium thermodynamics}, given by the variational condition
\begin{align*} 			&\delta\sum_{k=0}^{N-1}\mathsf{\Lag}_d(\Xi_k,a_k,S_k,S_{k+1})\\
			&\quad\quad\quad+\sum_{k=0}^{N-1}\big\langle\mathsf{F}_d^{\text{ext}-}(\Xi_k,a_k,S_k,S_{k+1}),\eta_k\big\rangle
			+\big\langle\mathsf{F}_d^{\text{ext}+}(\Xi_k,a_k,S_k,S_{k+1}),\eta_{k+1}\big\rangle=0,
\end{align*}
subject to discrete variational and phenomenological constraints
\begin{align*} 
			&D_3\mathsf{\Lag}_d(\Xi_k,a_k,S_k,S_{k+1})\delta S_k+D_4\mathsf{\Lag}_d(\Xi_k,a_k,S_k,S_{k+1})\delta S_{k+1}\\
			&\quad\quad\quad \quad \quad =\big\langle\mathsf{F}_d^{\text{fr}-}(\Xi_k,a_k,S_k,S_{k+1}),\eta_k\big\rangle
			+\big\langle\mathsf{F}_d^{\text{fr}+}(\Xi_k,a_k,S_k,S_{k+1}),\eta_{k+1}\big\rangle,\\[2mm]
			&\mathsf{P}_d(\Xi_k,a_k,S_k,S_{k+1})=0,
\end{align*}
and the discrete Euler-Poincar\'e constraints
		\[
			\delta\Xi_k=-T_eR_{\Xi_k}\eta_k+T_eL_{\Xi_k}\eta_{k+1},\quad\delta a_k=-(\eta_k)_M a_k,
		\]
		where $\eta_d$ is any discrete curve in $\LieAlgebraFont{g}$ with vanishing endpoints.
		
		\item\label{eq:DTEP} The discrete curve $(\Xi_d,a_d, S_d)$ on $G \times \operatorname{Orb}(a_{\rm ref}) \times \mathbb{R}  $ is solution of the \emph{discrete Euler-Poincar\'e equations for simple thermodynamical systems}		\begin{empheq}[left=\empheqlbrace]{align}
			& D_1\mathsf{\Lag}_d(\Xi_k,a_k,S_k,S_{k+1})T_eR_{\Xi_k}\notag\\
			&\quad \quad =D_1\mathsf{\Lag_d}(\Xi_{k-1},a_{k-1},S_{k-1},S_k)T_eL_{\Xi_{k-1}}
			-\mathbf{J}\big(D_2\mathsf{\Lag}_d(\Xi_k,a_k,S_k,S_{k+1})\big)\label{eq:discrete-TEP-equation}\\
			&\quad \quad \;\;\;\;+(\mathsf{F}_d^{\text{ext}-}+\mathsf{F}_d^{\text{fr}-})(\Xi_k,a_k,S_k,S_{k+1})+(\mathsf{F}_d^{\text{ext}+}+\mathsf{F}_d^{\text{fr}+})(\Xi_{k-1},a_{k-1},S_{k-1},S_{k}),\notag\\[2mm]
			& \mathsf{P}_d(\Xi_k,a_k,S_k,S_{k+1})=0,\\[2mm]
			& a_{k+1}=\Xi_k^{-1}a_k\label{eq:discrete-advection-equation}.
		\end{empheq}
	\end{enumerate}
\end{theorem}
\begin{proof}
	The equivalence between (1) and (2) is given in \citet[Theorem 3.8]{GBYo2017a}. Showing that (3) and (4) are equivalent is done as usual; note that the reconstruction equation comes from the definition of $\Xi_k$ and that the advection equation comes from the definition of $a_k$:
	\begin{equation*}
		a_{k+1}=g_{k+1}^{-1}a_\text{ref}=g_{k+1}^{-1}g_kg_k^{-1}a_\text{ref}=\Xi_k^{-1}a_k.
	\end{equation*}
	
	It remains to show that (1) and (3) are equivalent. Firstly, given a variation $\delta g_k$ of $g_k$,  we have by setting $\eta_k=g_k^{-1}\delta g_k=T_eL_{g_k^{-1}}\delta g_k$ that
	\begin{align*}
		\delta\Xi_k&=-g_k^{-1}\delta g_k g_k^{-1}g_{k+1}+g_k^{-1}g_{k+1}g_{k+1}^{-1}\delta g_{k+1}\\
		&=-T_eR_{\Xi_k}\eta_k+T_eL_{\Xi_k}\eta_{k+1}.
	\end{align*}
	Secondly we find that $\delta a_k=-(\eta_k)_Ma_k$ using a computation similar to the one we did in the proof of Theorem \ref{thm:TEP-reduction}.
	
	Conversely, suppose that we are given the curve $\Xi_d$ and a variation $\delta\Xi_d$. We want to find a discrete curve $g_d$ and a discrete variation $\delta g_d$ of $g_d$ starting from curves $\Xi_d$ and $\eta_d$ as above. This is achieved by computing successively $g_{k+1}=g_k\Xi_k$, and by setting $\delta g_k=g_k^{-1}\eta_k$. Since $\eta_d$ is arbitrary and zero at endpoints, $\delta g_d$ is arbitrary and zero at endpoints. We conclude using the left $G_{a_\text{ref}}$-invariance (respectively equivariance) to obtain the variational principle (1) from the variational principle (3).
\end{proof}

\subsection{Discrete Kelvin-Noether theorem}\label{sec:discrete-Kelvin-Noether}

We use the notations of Section \ref{sec:Kelvin-Noether}. As a discrete analogue of the Kelvin-Noether quantity $I$, we consider the map $I_d:\mathcal{C}\times G\times  \operatorname{Orb}(a_{\rm ref}) \times\R\times\R\to\R$ defined by:
\[
I_d(c,\Xi,a,S_0,S_1)=\big\langle\mathcal{K}(c,a),D_1\mathsf{\Lag}_d(\Xi,a,S_0,S_1)T_eR_{\Xi}\big\rangle,
\]
for all $c\in\mathcal{C}$, $\Xi\in G$, $a\in  \operatorname{Orb}(a_{\rm ref}) $ and $S_0$, $S_1\in\R$.

\begin{corollary}[\textbf{Discrete Kelvin-Noether theorem}]\label{cor:discrete-Kelvin-Noether}
	Let $c_\text{ref}\in\mathcal{C}$ fixed and let $(\Xi_d,a_d,S_d)$ be a discrete curve solution of the equations \eqref{eq:discrete-TEP-equation}--\eqref{eq:discrete-advection-equation}. Define $c_k=g_k^{-1}c_\text{ref}$ and $I_k=I_d(c_k,\Xi_k,a_k,S_k,S_{k+1})$. Then, for any $k\in\{1,\dots,N-1\}$, we have
	\begin{align*}
		I_k-I_{k-1}&=\big\langle\mathcal{K}(c_k,a_k),-\mathbf{J}\big(D_2\Lag_d(\Xi_k,a_k,S_k,S_{k+1})\big)\\
		&\quad\quad+(F_d^{\text{ext}-}+F_d^{\text{fr}-})(\Xi_k,a_k,S_k,S_{k+1})+(F_d^{\text{ext}+}+F_d^{\text{fr}+})(\Xi_{k-1},a_{k-1},S_{k-1},S_{k})\big\rangle.
	\end{align*}
\end{corollary}

\begin{proof}
	From the equivariance property of $\mathcal{K}$ we deduce that $\mathcal{K}(c_{k-1},a_{k-1})=\Ad_{\Xi_k}\mathcal{K}(c_k,a_k)$. Therefore we obtain:
	\begin{equation*}
		I_k-I_{k-1}=\big\langle\mathcal{K}(c_k,a_k),D_1\mathsf{\Lag}_d(\xi_k,a_k,S_k,S_{k+1})T_eR_{\Xi_k}\big\rangle-\big\langle\mathcal{K}(c_k,a_k), D_1\mathsf{\Lag}_d(\Xi_{k-1},a_{k-1},S_{k-1},S_k)T_eL_{\Xi_{k-1}}\big\rangle,
	\end{equation*}
	from which we conclude using equation \eqref{eq:discrete-TEP-equation}.
\end{proof}

\subsection{Group difference maps and reformulation of the discrete evolution equations}\label{sec:group-difference-map}

As can be seen from Theorem \ref{thm:DTEP-reduction}, the reduced discrete evolution equations now take place on $G \times \operatorname{Orb}(a_0)\times \mathbb{R} $, considered here as the discrete reduced extended tangent space $(G\times G \times\mathbb{R})/G_{a_{\rm ref}}$. However, numerically speaking, solving differential equations on manifolds is more difficult that solving differential equations on vector spaces, as it is difficult to design a numerical scheme which ensures that the discrete evolution actually takes place in the manifold. The aim of this section is to transport the equations obtained in Theorem \ref{thm:DTEP-reduction} to the Lie algebra $\LieAlgebraFont{g}$. We will do so by using a \emph{group difference map} as introduced in \citet[Section 4]{BRMa2008}. Simply put, these maps are approximations of the exponential map $\exp:\LieAlgebraFont{g}\to G$, but that still share its main algebraic properties.

\begin{definition}[\textbf{Group difference map}]
	Let $G$ be a Lie group and denote by $\LieAlgebraFont{g}$ its Lie algebra. A \emph{group difference map} is a local diffeomorphism $\tau:\LieAlgebraFont{g}\to G$ mapping a neighborhood $\mathcal{N}_0$ of $0\in\LieAlgebraFont{g}$ to a neighborhood of $e\in G$, and such that $\tau(0)=e$ and $\tau(\xi)^{-1}=\tau(-\xi)$, for any $\xi\in\mathcal{N}_0$.
\end{definition}

\begin{definition}[\textbf{Right trivialized tangent of a group difference map}]
	Let $G$ be a Lie group, $\LieAlgebraFont{g}$ its Lie algebra and $\tau:\LieAlgebraFont{g}\to G$ a group difference map. The \emph{right trivialized tangent} of $\tau$ is the map $\mathrm{d}\tau:\LieAlgebraFont{g}\times\LieAlgebraFont{g}\to\LieAlgebraFont{g}$ defined by
	\[
		\mathbf{D}\tau(\xi)(\delta)=T_eR_{\tau(\xi)}\mathrm{d}\tau_\xi(\delta),
	\]
	for all $\xi$, $\delta\in\LieAlgebraFont{g}$. The \emph{inverse right trivialized tangent} of $\tau$ is the map $\mathrm{d}\tau^{-1}:\LieAlgebraFont{g}\times\LieAlgebraFont{g}\to\LieAlgebraFont{g}$ defined by
	\[
		\mathbf{D}\tau^{-1}(\tau(\xi))(\delta)=\mathrm{d}\tau^{-1}_\xi(T_eR_{\tau(-\xi)}\delta),
	\]
	for all $\xi$, $\delta\in\LieAlgebraFont{g}$. Thus $\mathrm{d}\tau_\xi(\mathrm{d}\tau^{-1}_\xi(\delta))=\delta$, for all $\xi$, $\delta\in\LieAlgebraFont{g}$. Note that $\mathrm{d}\tau$ and $\mathrm{d}\tau^{-1}$ are always linear in their second argument, but not necessarily in the first.
\end{definition}

\begin{proposition}[{\citet[Section 4]{BRMa2008}}]\label{pro:group-difference-map}
	Let $G$ be a Lie group, $\LieAlgebraFont{g}$ its Lie algebra and $\tau:\LieAlgebraFont{g}\to G$ a group difference map. The right trivialized tangent $\mathrm{d}\tau$ of $\tau$ satisfies the following properties:
	\begin{enumerate}[label=(\arabic*), ref=\thetheorem.(\arabic*)]
		\item\label{eq:RTT1} $\mathrm{d}\tau_\xi(\delta)=\Ad_{\tau(\xi)}\mathrm{d}\tau_{-\xi}(\delta)$,
		\item\label{eq:RTT2} $\mathrm{d}\tau_{-\xi}^{-1}(\Ad_{\tau(-\xi)}\delta)=\mathrm{d}\tau_\xi^{-1}(\delta)$.
	\end{enumerate}
\end{proposition}
\begin{proof}
	For the first property, for any $\xi\in\LieAlgebraFont{g}$ we have $\mu(\tau(\xi),\tau(-\xi))=e$, where $\mu$ denotes the multiplication law in $G$. Differentiating this relation we obtain for any $\delta\in\LieAlgebraFont{g}$:
	\[
		T_{\tau(\xi)}R_{\tau(-\xi)}\mathbf{D}\tau(\xi)(\delta)-T_{\tau(-\xi)}L_{\tau(\xi)}\mathbf{D}\tau(-\xi)(\delta)=0,
	\]
	and finally using the definition of the right trivialized tangent of $\tau$ we obtain
	\[
		\mathrm{d}\tau_\xi(\delta)=T_{\tau(-\xi)}L_{\tau(\xi)}T_eR_{\tau(-\xi)}\mathrm{d}\tau_{-\xi}(\delta).
	\]
	The second property results from an application of the first one, with $\delta$ replaced by $\mathrm{d}\tau^{-1}_\xi(\delta)$.
\end{proof}

Approximations of the exponential map are available in terms of rational fractions, these are the well-known Pad\'e approximants of the exponential. The $(1,1)$ Pad\'e approximant of the exponential is also known as the \emph{Cayley map}, and is widely used in computational geometric mechanics. For more details, see \citet[Section III.4.1 and IV.8.3]{HLW2006} and \citet[Section 4.6]{BRMa2008}.

Let $\tau:\LieAlgebraFont{g}\to G$ be a group difference map and $h>0$ be a time step. We are now going to transport the reduced discrete variational formulation \ref{eq:DRTVP} and the associated discrete equations \ref{eq:DTEP} to the Lie algebra $\LieAlgebraFont{g}$ using this group difference map $\tau$. We define a new discrete Lagrangian $\lag_d:\LieAlgebraFont{g}\times\Orb(a_\text{ref})\times\R\times\R\to\R$ by
\[
	\lag_d(\xi_k,a_k,S_k,S_{k+1})=\mathsf{\Lag}_d(\Xi_k,a_k,S_k,S_{k+1}),\text{ with }\xi_k=\frac{1}{h}\tau^{-1}(\Xi_k).
\]
This can simply be considered as a change of variable. This definition naturally extends to other quantities: the external and friction forces as well as the map giving the discrete phenomenological constraints; we obtain maps $f_d^{\text{ext}\pm}$, $f_d^{\text{fr}\pm}:\LieAlgebraFont{g}\times\Orb(a_\text{ref})\times\R\times\R\to\LieAlgebraFont{g}^*$ and $p_d:\LieAlgebraFont{g}\times\Orb(a_\text{ref})\times\R\times\R\to\R$. Therefore, using Proposition \ref{pro:group-difference-map}, the discrete variational formulation \ref{eq:DRTVP} can be reformulated for a curve $(\xi_d,a_d,S_d)$ as:
\begin{equation}\label{eq:Variat_Cond}
\delta\sum_{k=0}^{N-1}\lag_d(\xi_k,a_k,S_k,S_{k+1})+\sum_{k=0}^{N-1}\big\langle f_d^{\text{ext}-}(\xi_k,a_k,S_k,S_{k+1}),\eta_k\big\rangle
	+\big\langle f_d^{\text{ext}+}(\xi_k,a_k,S_k,S_{k+1}),\eta_{k+1}\big\rangle=0,
\end{equation} 
subject to the discrete variational and phenomenological constraints
\begin{equation}\label{eq:Variat_Const} 
\begin{aligned}
&D_3\lag_d(\xi_k,a_k,S_k,S_{k+1})\delta S_k+D_4\lag_d(\xi_k,a_k,S_k,S_{k+1})\delta S_{k+1}\\
	&\quad\quad\quad \quad \quad =\big\langle f_d^{\text{fr}-}(\xi_k,a_k,S_k,S_{k+1}),\eta_k\big\rangle
	+\big\langle f_d^{\text{fr}+}(\xi_k,a_k,S_k,S_{k+1}),\eta_{k+1}\big\rangle,\\[2mm]
	&p_d(\xi_k,a_k,S_k,S_{k+1})=0,
\end{aligned}
\end{equation}  
and the discrete Euler-Poincar\'e constraints
\begin{equation}\label{eq:EP_Const}
	\delta\xi_k=-\frac{1}{h}\mathrm{d}\tau^{-1}_{h\xi_k}(\eta_k)+\frac{1}{h}\mathrm{d}\tau^{-1}_{-h\xi_k}(\eta_{k+1}),\quad\delta a_k=-(\eta_k)_M a_k,
	\end{equation} 
where $\eta_d$ is any discrete curve in $\LieAlgebraFont{g}$ with vanishing endpoints. By applying the discrete variational formulation \eqref{eq:Variat_Cond}--\eqref{eq:EP_Const} we get the following reformulation of equations \ref{eq:DTEP}:
\begin{empheq}[left=\empheqlbrace]{align*}
	& (\mathrm{d}\tau^{-1}_{h\xi_k})^*D_1\lag_d(\xi_k,a_k,S_k,S_{k+1})\\
	& \quad \quad =(\mathrm{d}\tau^{-1}_{-h\xi_{k-1}})^*D_1\lag_d(\xi_{k-1},a_{k-1},S_{k-1},S_k)
	-h\mathbf{J}\big(D_2\lag_d(\xi_k,a_k,S_k,S_{k+1})\big)\\
	&\quad \quad \;\;\;\;+h(f_d^{\text{ext}-}+f_d^{\text{fr}-})(\xi_k,a_k,S_k,S_{k+1})+h(f_d^{\text{ext}+}+f_d^{\text{fr}+})(\xi_{k-1},a_{k-1},S_{k-1},S_{k}),\\[2mm]
	& p_d(\xi_k,a_k,S_k,S_{k+1})=0,\\[2mm]
	& a_{k+1}=\tau(-h\xi_k)a_k.
\end{empheq}

\begin{remark}
	For the kind of variational integrators that we have presented, if we want to define energy properly at the discrete level, then the time step $h$ has to be promoted to a full dynamic variable as well, meaning that the {discrete} Lagrangian of the system {is interpreted as being time-dependent} and thus the {discrete} Euler-Lagrange equations decompose into the usual dynamical part and an equation enforcing energy conservation in addition. See \cite{KMO1999} and \cite{LeDi2002} for such an approach in the case of mechanical systems without thermal effects.
\end{remark}

Finally the Kelvin-Noether quantity introduced in Section \ref{sec:discrete-Kelvin-Noether} can be reformulated as:
\begin{equation}\label{I_d}  
	I_d(c,\xi,a,S_0,S_1)=\big\langle\mathcal{K}(c,a),(\mathrm{d}\tau^{-1}_{h\xi})^*D_1\lag_d(\xi,a,S_0,S_1)\big\rangle,
\end{equation} 
for any $c\in\mathcal{C}$, $\xi\in\LieAlgebraFont{g}$, $a\in\Orb(a_\text{ref})$, $S_0, S_1\in\R$; and the discrete Kelvin-Noether Theorem \ref{cor:discrete-Kelvin-Noether} now reads:
\begin{equation}\label{DKN}
\begin{aligned}
I_k-I_{k-1}&=\big\langle\mathcal{K}(c_k,a_k),-h\mathbf{J}\big(D_2\lag_d(\xi_k,a_k,S_k,S_{k+1})\big)\\
	&\quad\quad+h(f_d^{\text{ext}-}+f_d^{\text{fr}-})(\xi_k,a_k,S_k,S_{k+1})+h(f_d^{\text{ext}+}+f_d^{\text{fr}+})(\xi_{k-1},a_{k-1},S_{k-1},S_{k})\big\rangle.
\end{aligned} 
\end{equation} 
\section{An heavy top in Stokes flow}\label{sec:ball}

We now illustrate the variational discretization developed above with the example of a heavy top moving in a Stokes flow.
In the situation we consider, the motion of the top completely determines the motion of the Stokes flow, through the no-slip boundary condition at the fluid-body interface, as determined in \cite{La1975}, \cite{BrHa1983}, \cite{KiKa1991}. The torque exerted by the viscous fluid is interpreted as a friction force responsible for entropy production, at the origin of the irreversible character of the system. Since the system is isolated, the total energy, composed of the mechanical energy and the internal energy, is conserved. Note however that this conservation law is not due to the existence of a Hamiltonian structure, but is rather the reflection of the first law of thermodynamics applied to the system. As we will observe, our numerical scheme reproduces notably this energy conservation at the discrete level, {whereas a standard discretization, possibly of higher order, will not do so, in general}.

\begin{remark} 
	The considered example here is a toy model whose main purpose is to illustrate our integrator in the simplest possible situation, and serves as the basis for forthcoming developments.
\end{remark}

\subsection{The system and its variational formulation}

We consider a simplified example of a heavy top rotating in a viscous fluid modeled by a Stokes flow. The top is composed of a ball of radius $a$, total mass $m$ and moment of inertia tensor $\mathbb{I}$ (diagonalized in the body principal axes). This rigid body will be denoted by $\mathcal{B}$, and its boundary $\partial\mathcal{B}$ is the sphere bounding the ball. Let $(C,\mathbf{e}_1,\mathbf{e}_2,\mathbf{e}_3)$ be the canonical orthonormal frame of $\R^3$, $C$ being the \emph{fixed} geometric center of the ball, around which the ball rotates. The center of mass of the ball will be denoted by $G$; in the case where $\mathbb{I}$ is not proportional to the identity, then $G$ does not coincide with $C$.

Forgetting gravity for the moment, remembering that the center of the ball $C$ is fixed, the configuration space of $\mathcal{B}$ is the Lie group $G=\SO(3)$, so a configuration of $\mathcal{B}$ is a rotation matrix $R$ in $\R^3$ and the kinematics of $\mathcal{B}$ is given by a curve $t\in[0,T]\mapsto R(t)\in\SO(3)$. We recall a few elementary facts concerning the Lie algebra $\so(3)$ of the configuration space:
\begin{itemize}
	\item The vector spaces $\so(3)$ and $\R^3$ are isomorphic, the isomorphism being given by the hat map $\hat{}:\R^3\mapsto\so(3)$,
	\begin{equation}\label{hat_map} 
		\mathbf{W}=(W_1,W_2,W_3)\mapsto\hat{\mathbf{W}}=\mathbf{W}\times\cdot=
		\begin{pmatrix}
			0 & -W_3 & W_2\\
			W_3 & 0 & -W_1\\
			-W_2 & W_1 & 0
		\end{pmatrix}.
	\end{equation} 
	
	\item The Lie algebras $\so(3)$, with matrix commutator $[\cdot,\cdot]$, and $\R^3$, with cross product $\times$, are isomorphic, meaning that $\widehat{\mathbf{V}\times\mathbf{W}}=[\hat{\mathbf{V}},\hat{\mathbf{W}}]$, for all $\mathbf{V}$, $\mathbf{W}\in\R^3$.
	
	\item We endow $\R^3$ with the usual inner product $\mathbf{V}\cdot\mathbf{W}=\mathbf{V}^\mathsf{T}\mathbf{W}$ and $\so(3)$ with the inner product $\langle\hat{\mathbf{V}},\hat{\mathbf{W}}\rangle=\frac{1}{2}\Tr(\hat{\mathbf{V}}^\mathsf{T}\hat{\bm{W}})$. Then the hat map is isometric: $\langle\hat{\mathbf{V}},\hat{\mathbf{W}}\rangle=\mathbf{V}\cdot\mathbf{W}$, for all $\mathbf{V}$, $\mathbf{W}\in\R^3$.
	\item If the Lie group $\SO(3)$ acts on $\R^3$ by left multiplication and on $\so(3)$ by the adjoint representation which is matrix conjugation, then the hat map is equivariant: $\widehat{R\mathbf{W}}=R\hat{\mathbf{W}}R^{-1}$, for all $R\in\SO(3)$ and $\mathbf{W}\in\R^3$.
\end{itemize}

If $\mathbf{X}\in\mathcal{B}$ is a point of the \emph{reference configuration}, $\mathbf{X}$ is transformed after $t$ units of time into a point $\mathbf{x}(t)=R(t)\mathbf{X}\in R(\mathcal{B})$ of the \emph{actual configuration}. The \emph{material velocity} of $\mathbf{X}$ (in the reference configuration) is $\dot{R}\mathbf{X}$, whereas the \emph{spatial velocity} of $\mathbf{x}$ (in the actual configuration) is $\dot{R}R^{-1}\mathbf{x}$. Since $R\in\SO(3)$, we can write $\dot{R}R^{-1}=\hat{\bm{\omega}}$ for the curve $t\in[0,T]\mapsto\bm{\omega}(t)\in\R^3$ called the \emph{spatial angular velocity} of $\mathcal{B}$. Let $(C,\mathbf{b}_1,\mathbf{b}_2,\mathbf{b}_3)$ be the so-called \emph{body frame}, that is, the orthonormal frame associated to $(C,\mathbf{e}_1,\mathbf{e}_2,\mathbf{e}_3)$ that moves according to the motion of $\mathcal{B}$, meaning that $\mathbf{b}_i(t)=R(t)\mathbf{e}_i$, for $t\in[0,T]$ and $i\in\{1,2,3\}$. The \emph{body angular velocity} is $\bm{\Omega}=R^{-1}\bm{\omega}$, or $\hat{\bm{\Omega}}=R^{-1}\dot{R}$ in matrix terms. For more details about rigid bodies, see \citet[Chapter 15]{MaRa1999}. 

The kinetic energy of $\mathcal{B}$ is given in the material formalism by
\[
	K_\mathcal{B}(R,\dot{R})=\frac{\rho}{2}\int_{\mathcal{B}}\|\dot{R}\mathbf{X}\|^2\,\mathrm{d}\mathbf{X},
\]
for any $(R,\dot{R})\in T\SO(3)$, and where $\rho$ is the volumetric mass density of $\mathcal{B}$. The kinetic energy is clearly $\SO(3)$-invariant.

In order to take into account gravity, we introduce $\bm{\chi}$ as the unit vector of the line going from the fixed point $C$ to the center of mass $G$. The potential energy of $\mathcal{B}$ is given in the material formalism by
\[
	V_\mathcal{B}(R)=mg\ell\mathbf{e}_3\cdot R\bm{\chi}=mg\ell R^{-1}\mathbf{e}_3\cdot\bm{\chi},
\]
where $g$ is the gravitational acceleration constant, and $\ell\bm{\chi}$ is the vector from the fixed point of the top to its center of mass (at time zero). However, we see that the $\SO(3)$-invariance is broken upon introducing gravity, as the potential energy is only invariant with respect to rotations that preserve $\mathbf{e}_3$, resulting in a $\SO(2)$-invariance instead. Thus, instead of the $\SO(3)$-invariance, we will consider invariance relatively to $\SO(3)_{\mathbf{e}_3}$ the subgroup of $\SO(3)$ that preserve $\mathbf{e}_3$, and consider $\mathbf{e}_3$ as a \emph{parameter} of the full Lagrangian (so we choose $M=\R^3$ for the space of advected parameters).


The environment of the heavy top is modeled by an unbounded fluid $\mathcal{F}$, supposed to be of constant volumetric mass density $\rho_\mathcal{F}$, incompressible, Newtonian with dynamic viscosity $\mu$, and at rest far from the top. Its flow is governed by the Navier-Stokes equations, however, we will make use of the Stokes approximation, that permits us to find an analytical expression for the velocity field of the fluid as well as the total force and torque exerted by the fluid on the heavy top in this particular setting of a spherical geometry. Let $\mathbf{u}(t,\mathbf{x})$ be the velocity field of the fluid and $\bar{\sigma}=-\bar{p}\Id_3+2\mu D(\mathbf{u})$ its Cauchy stress tensor, $\bar{p}=p+\rho_\mathcal{F}gz$ being the pressure and $D(\mathbf{u})$ being the strain rate tensor. Then the Stokes equations together with the appropriate boundary conditions read:
\[
	\begin{cases}
		\operatorname{div}\bar{\sigma}=0\Longleftrightarrow\Delta\mathbf{u}(t,\mathbf{x})=\frac{1}{\mu}\nabla\bar{p}(t,\mathbf{x})\\
		\nabla\cdot\mathbf{u}(t,\mathbf{x})=0\\
		\mathbf{u}(t,\mathbf{x})=\bm{\omega}(t)\times\mathbf{x}\text{ for any }\mathbf{x}\in\partial\mathcal{B}\\
		\mathbf{u}(t,x)\longrightarrow 0\text{ when }\|\mathbf{x}\|\to+\infty
	\end{cases}.
\]
Recall that the Stokes approximation is a quasi-stationary approximation and that physically speaking, this approximation is valid at low Reynolds numbers only. Let $\mathbf{f}(t)$ be the total force and $\bm{\tau}_C(t)$ the total torque with respect to the origin $C$ exerted on the sphere by the fluid, in the spatial formalism. Denoting by $\partial\mathcal{B}^+$ the boundary sphere oriented by the unit normal $\mathbf{n}^+$ pointing from the ball towards the fluid, these are defined by:
\[
	\mathbf{f}=\int_{\partial\mathcal{B}^+}\bar{\sigma}\cdot\mathbf{n}^+\,\mathrm{d}S,\quad\bm{\tau}_C=\int_{\partial\mathcal{B}^+}\mathbf{x}\times(\bar{\sigma}\cdot\mathbf{n}^+)\,\mathrm{d}S.
\]
Using the general solution of Lamb for a spherical coordinate system, as well as the boundary condition $\mathbf{u}(t,\mathbf{x})=\bm{\omega}(t)\times\mathbf{x}$ at any spatial point $\mathbf{x}$ of $\partial\mathcal{B}$, we can compute that (see \citet[Chapter 3]{BrHa1983}, \citet[Example 4.2]{KiKa1991}, or \citet[Article 337]{La1975} for the full computation, which is not straightforward):
\[
	\mathbf{u}(t,\mathbf{x})=\frac{a^3}{\|\mathbf{x}\|^3}\bm{\omega}(t)\times\mathbf{x},\;\;{\text{for all $|\mathbf{x}|\geq a$}} ,
	\qquad\mathbf{f}(t)=0,
	\qquad\bm{\tau}_C(t)=-8\pi\mu a^3\bm{\omega}(t).
\]
Note that since $\mathbf{f}=0$, the total torque with respect to the center of mass $G$ is equal to $\bm{\tau}_C$ and will be denoted simply by $\bm{\tau}$. In the material formalism we will denote this torque by $F_{\mathbf{e}_3}$, it is given by
\begin{equation}\label{eq:ball-friction}
	F_{\mathbf{e}_3}(R,\dot{R},S)=\hat{\bm{\tau}}R=-8\pi\mu a^3\dot{R}.
\end{equation}

This torque is the result of the viscosity of the fluid.
Thermodynamically, the \emph{simple} system we want to consider is composed of both the top and the fluid, therefore we introduce one entropy variable $S\in\R$ that describes the entropy of both the top and the fluid. Remembering that the kinetic energy of a fluid in Stokes flow is always neglected, the full $\SO(3)_{\mathbf{e}_3}$-invariant Lagrangian of the system $\Lag_{\mathbf{e}_3}:T\SO(3)\times\R\to\R$ is given by
\begin{equation}\label{eq:ball-lagrangian}
	\Lag_{\mathbf{e}_3}(R,\dot{R},S)=K_\mathcal{B}(R,\dot{R})-V_\mathcal{B}(R)-U_{\mathcal{B}}(S),
\end{equation}
where $U_\mathcal{B}(S)$ denotes the internal energy of the top $\mathcal{B}$, which we will make more explicit later. The reason why the internal energy of the fluid is neglected is because the change in the internal energy of the fluid happens only locally, around the ball, whereas the fluid is considered as infinite. The temperature of the system will be denoted by $T=-\frac{\partial\Lag_{\mathbf{e}_3}}{\partial S}=\frac{\partial U_\mathcal{B}}{\partial S}$; it is the temperature of the top as well as of the fluid, in that particular simplified model. Additionally, there are no external forces, but there is a friction force already mentioned above and that will be the sole friction force acting on the system, and therefore the one that creates entropy. Note also that since the system we consider is composed of both the top and the fluid, there is no external heat transfer. Therefore, the full variational formulation \ref{eq:TVP} in the material formalism reads:
\begin{equation}\label{eq:ball-TVP}
		\delta\int_0^T\left[\int_\mathcal{B}\|\dot{R}\mathbf{X}\|^2\mathrm{d}\mathbf{X}-mg\ell R^{-1}\mathbf{e}_3\cdot\bm{\chi}-U_\mathcal{B}(S)\right]\mathrm{d}t=0,
\end{equation}
subject to the variational and phenomenological constraints
\[
\delta S=\frac{8\pi\mu a^3}{T}\big \langle \dot{R},\delta R \big\rangle \quad\text{and}\quad \dot{S}=\frac{8\pi\mu a^3}{T} \big\langle \dot{R},\dot{R} \big\rangle .
\]

Now we can proceed to the reduction of the system as developed in Section \ref{sec:TEP-reduction}. The kinetic energy of the top is reduced to 
\[
	\frac{\rho}{2}\int_\mathcal{B}\|R^{-1}\dot{R}\mathbf{X}\|^2\,\mathrm{d}\mathbf{X}=\frac{\rho}{2}\int_\mathcal{B}\|\bm{\Omega}\times\mathbf{X}\|^2\,\mathrm{d}\mathbf{X}=\frac{1}{2}\mathbb{I}\bm{\Omega}\cdot\bm{\Omega},
\]
where $\mathbb{I}$ is the inertia tensor of the body. Introducing the reduced variable $\bm{\Gamma}=R^{-1}\mathbf{e}_3$ for $\mathbf{e}_3$, which can be interpreted as the direction of gravity as seen from the heavy top, we obtain from \eqref{eq:ball-lagrangian} the reduced Lagrangian $\lag:\so(3)\times\Orb(\mathbf{e}_3)\times\R\to\R$ given by:
\[
	\lag(\hat{\bm{\Omega}},\bm{\Gamma},S)=\frac{1}{2}\mathbb{I} {\bm{\Omega}}\cdot {\bm{\Omega}}-mg\ell\bm{\Gamma}\cdot\bm{\chi}-U_\mathcal{B}(S).
\]
The friction force due to viscosity now reads: 
\[
	f(\hat{\bm{\Omega}},\bm{\Gamma},S)
	=F_{\mathbf{e}_3}(e,\hat{\bm{\Omega}},S)
	=R^{-1}F_{\mathbf{e}_3}(R,\dot{R},S)
	=-8\pi\mu a^3\hat{\bm{\Omega}}.
\]
Hence the reduced variational formulation \ref{eq:RTVP} in body coordinates is:
\[
\delta \int_0^T\left[\frac{1}{2}\mathbb{I}\bm{\Omega}\cdot\bm{\Omega}-mg{\ell}\bm{\Gamma}\cdot\bm{\chi}-U_\mathcal{B}(S)\right]\mathrm{d}t=0,
\]
subject to the variational and phenomenological constraints
\[
\delta S=\frac{8\pi\mu a^3}{T}\bm{\Omega}\cdot\bm{\Sigma}, \qquad \dot{S}=\frac{8\pi\mu a^3}{T}\|\bm{\Omega}\|^2,
\]
and the Euler-Poincar\'e constraints
\[
\delta\bm{\Omega}=\dot{\bm{\Sigma}}+\bm{\Omega}\times\bm{\Sigma},\quad\delta\bm{\Gamma}=-\bm{\Sigma}\times\bm{\Gamma}, 
\]
where $\bm{\Sigma}:[0,T]\to\R^3$ is an arbitrary curve {vanishing at $t=0,T$}.

This variational principle yields, after extremizing the action functional above, the reduced equations \ref{eq:TEP-equations} in body coordinates for our system. Introducing the angular momentum in body coordinates $\bm{\Pi}=\widehat{\mathbb{I}\bm{\Omega}}$, these equations are:
\begin{equation}\label{eq:ball-TEP}
	\everymath{\displaystyle}
	\begin{cases}
		\dot{\bm{\Pi}}+\bm{\Omega}\times\bm{\Pi}=-8\pi\mu a^3\bm{\Omega}+mg\ell\bm{\Gamma}\times\bm{\chi}\\[0.2cm]
		\dot{S}=\frac{8\pi\mu a^3}{T}\|\bm{\Omega}\|^2\\[0.2cm]
		\dot{\bm{\Gamma}}=-\bm{\Omega}\times\bm{\Gamma}
	\end{cases}.
\end{equation}

We will assume that the internal energy ${U}_\mathcal{B}(S)$ of the rigid body follows the Dulong-Petit law (see \cite{PeDu1819}): ${U}_\mathcal{B}(S)=3N_0RT(S)$, where $N_0$ is the number of moles of the rigid body and $R$ is the universal gas constant. {From this expression and the definition of temperature we deduce that the rigid body temperature depends on its entropy in the following way:
\[
T=T_0\exp\left(\frac{S-S_0}{3N_0R}\right).
\]
}

The total energy of the system is given in body coordinates by:
\[
	e(\hat{\bm{\Omega}},\bm{\Gamma},S)=\frac{1}{2}\mathbb{I}\bm{\Omega}\cdot\bm{\Omega}+mg\ell\bm{\Gamma}\cdot\bm{\chi}+U_\mathcal{B}(S)
\]
and the energy balance is simply $\frac{\mathrm{d}e}{\mathrm{d}t}=0$ (see Corollary \ref{cor:energy-balance}). Finally concerning the Kelvin-Noether quantity (see Corollary \ref{cor:Kelvin-Noether}), we choose $\mathcal{C}=\so(3)$ and $\mathcal{K}:\so(3)\times\R^3\to\so(3)$, $(\hat{\mathbf{W}},\hat{\bm{\Gamma}})\mapsto\hat{\mathbf{W}}$. Then the Kelvin-Noether theorem in the particular case where $\mathbf{W}=\bm{\Gamma}$ yields the rate of change of the spatial angular momentum $\bm{\pi}=R\bm{\Pi}$ of the heavy top:
\begin{equation}\label{KN_example}
{\frac{\mathrm{d}\pi_3}{\mathrm{d}t}= \frac{\mathrm{d}}{\mathrm{d}t}\mathbb{I}  \boldsymbol{\Omega}\cdot\boldsymbol{\Gamma} = -8\pi\mu a^3 \boldsymbol{\Gamma}\cdot \boldsymbol{\Omega}}.
\end{equation}

\subsection{Variational discretization}

We will work directly from the setting introduced in Section \ref{sec:group-difference-map}. Let $h>0$ be the time step. We need to choose a group difference map $\tau$, a finite difference map $\varphi$, and use these to build the discrete Lagrangian $\lag_d$, the discrete friction forces $f^{\pm}_d$, and the discrete phenomenological constraint $p_d$. The curve $t\mapsto R(t)\in\SO(3)$ is discretized into a sequence $R_k\in\SO(3)$, $k\in\{0,\dots,N\}$. The intermediary variable that was previously denoted by $\Xi_k$ is actually $R_k^{-1}R_{k+1}$. The body angular velocity curve $t\mapsto\hat{\bm{\Omega}}(t)\in\so(3)$ is discretized into a sequence $\hat{\bm{\Omega}}_k\in\so(3)$, $k\in\{0,\dots,N\}$; recall from Section \ref{sec:group-difference-map} that the $\hat{\bm{\Omega}}_k$ are related to the $R_k$ through the relation
\begin{equation}\label{def_omega_k} 
	\hat{\bm{\Omega}}_k=\frac{1}{h}\tau^{-1}(R_k^{-1}R_{k+1}).
\end{equation} 
The advected parameter curve $t\mapsto\bm{\Gamma}(t)\in\R^3$, is discretized into a sequence $\bm{\Gamma}_k\in\R^3$, and the entropy curve $t\mapsto S(t)$ is discretized into a sequence $S_k\in\R$, for $k\in\{0,\dots,N\}$.

For the group difference map, we choose $\tau=\operatorname{cay}:\so(3)\to\SO(3)$ the Cayley map given by (see \citet[IV.8.3]{HLW2006} for details):
\[
	\operatorname{cay}(\hat{\bm{\Omega}})=\bigg(\Id-\frac{\hat{\bm{\Omega}}}{2}\bigg)^{-1}\bigg(\Id+\frac{\hat{\bm{\Omega}}}{2}\bigg),
	\quad\mathrm{d}\operatorname{cay}^{-1}_{\hat{\bm{\Omega}}}({\hat{\bm{\Psi}}})=\bigg(\Id-\frac{\hat{\bm{\Omega}}}{2}\bigg)\hat{\bm{\Psi}}\bigg(\Id+\frac{\hat{\bm{\Omega}}}{2}\bigg).
\]
Remember that the group difference map is responsible for passing from the Lie algebra $\so(3)$ to the Lie group $\SO(3)$, and as such, constitutes an approximation of the exponential map. Also remark that since the Cayley map is expressed in matrix terms only, at the discrete level we will work on $\so(3)$ rather than $\R^3$ exclusively. 

For the finite difference map $\varphi:\SO(3)^2\times\R^2\to T\SO(3)\times T\R$ we choose:
\[
	\varphi(R_k,R_{k+1},S_k,S_{k+1})=\left(R_k,R_k\hat{\bm{\Omega}}_k,S_k,\frac{S_{k+1}-S_k}{h}\right),
\]
{where $\hat{\boldsymbol{\Omega} }_k $ is defined by \eqref{def_omega_k}.} 
For the discrete Lagrangian, we first define $\Lag_{d,\mathbf{e}_3}$ by setting $\Lag_{d,\mathbf{e}_3}=hL_{\mathbf{e}_3}\circ\varphi$, which simply reads:
\[
	\Lag_{d,\mathbf{e}_3}(R_k,R_{k+1},S_k,S_{k+1})=h\Lag_{\mathbf{e}_3}(R_k,R_k\hat{\bm{\Omega}}_k,S_k).
\]
After reduction we obtain:
\[
	\mathsf{\Lag}_d(R_k^{-1}R_{k+1},\bm{\Gamma}_k,S_k,S_{k+1})=h\Lag_{\mathbf{e}_3}(e,\hat{\bm{\Omega}}_k,S_k),
\]
which by definition yields:
\begin{equation}\label{discrete_Lagr} 
	\lag_d(\hat{\bm{\Omega}}_k,\bm{\Gamma}_k,S_k,S_{k+1})=h\lag(\hat{\bm{\Omega}}_k,\bm{\Gamma}_k,S_k)
	=\frac{h}{2}\big\langle\hat{\bm{\Omega}}_k,\widehat{\mathbb{I}{\bm{\Omega}}_k}\big\rangle-hmg\ell\langle\hat{\bm{\Gamma}}_k,\hat{\bm{\chi}}\rangle-hU_\mathcal{B}(S_k).
\end{equation} 

For the discrete friction forces $F_{d,\mathbf{e}_3}^\pm$ that result from the discretization of $F_{\mathbf{e}_3}$, first remember that we must have an approximation of the form:
\begin{align*}
	\int_{t_k}^{t_{k+1}}\left\langle F_{\mathbf{e}_3}(R(t),\dot{R}(t),S(t)),\delta R(t)\right\rangle\mathrm{d}t
	\approx\big\langle & F_{d,\mathbf{e}_3}^-(R_k,R_{k+1},S_k,S_{k+1}),\delta R_k\big\rangle\\
	&\quad+\big\langle F_{d,\mathbf{e}_3}^+(R_k,R_{k+1},S_k,S_{k+1}),\delta R_{k+1}\big\rangle.
\end{align*}
We choose to approximate the integral by the trapezoidal rule:
\begin{align*}
	\int_{t_k}^{t_{k+1}}\left\langle F_{\mathbf{e}_3}(R(t),\dot{R}(t),S(t)),\delta R(t)\right\rangle\mathrm{d}t
	\approx\frac{h}{2}\Big(\big\langle & F_{\mathbf{e}_3}(R(t_k),\dot{R}(t_k),S(t_k)),\delta R(t_k)\big\rangle\\
	&\quad+\big\langle F_{\mathbf{e}_3}(R(t_{k+1}),\dot{R}(t_{k+1}),S(t_{k+1})),\delta R(t_{k+1})\big\rangle\Big),
\end{align*}
and then the two terms on the right are approximated using the finite difference map $\varphi$:
\begin{gather*}
	F_{d,\mathbf{e}_3}^+(R_k,R_{k+1},S_k,S_{k+1})=\frac{h}{2}F_{\mathbf{e}_3}(R_{k+1},R_{k+1}\hat{\bm{\Omega}}_{k+1},S_{k+1})\in T^*_{R_{k+1}}\SO(3),\\
	F_{d,\mathbf{e}_3}^-(R_k,R_{k+1},S_k,S_{k+1})=\frac{h}{2}F_{\mathbf{e}_3}(R_{k},R_{k}\hat{\bm{\Omega}}_{k},S_{k})\in T^*_{R_{k}}\SO(3).
\end{gather*}
After reduction we obtain:
\begin{gather*}
	\mathsf{F}_d^+(R_k^{-1}R_{k+1},\bm{\Gamma}_k,S_k,S_{k+1})=\frac{h}{2}F_{\mathbf{e}_3}(e,\hat{\bm{\Omega}}_{k+1},S_{k+1})\in\LieAlgebraFont{g}^*,\\
	\mathsf{F}_d^-(R_k^{-1}R_{k+1},\bm{\Gamma}_k,S_k,S_{k+1})=\frac{h}{2}F_{\mathbf{e}_3}(e,\hat{\bm{\Omega}}_{k},S_{k})\in\LieAlgebraFont{g}^*,
\end{gather*}
which finally yields the following discrete friction forces:
\begin{equation}\label{f_fr_discret}
\begin{aligned}
f_d^+(\hat{\bm{\Omega}}_k,\bm{\Gamma}_k,S_k,S_{k+1})&=\frac{h}{2}f(\hat{\bm{\Omega}}_{k+1},\bm{\Gamma}_{k+1},S_{k+1})=-4\pi\mu a^3 h\hat{\bm{\Omega}}_{k+1},\\
f_d^-(\hat{\bm{\Omega}}_k,\bm{\Gamma}_k,S_k,S_{k+1})&=\frac{h}{2}f(\hat{\bm{\Omega}}_{k},\bm{\Gamma}_k,S_k)=-4\pi\mu a^3 h\hat{\bm{\Omega}}_{k}.
\end{aligned} 
\end{equation} 
The presence of $\hat{\bm{\Omega}}_{k+1}$ in $f_d^+(\hat{\bm{\Omega}}_k,\bm{\Gamma}_k,S_k,S_{k+1})$ will ultimately leads to an implicit integrator as we will see below. Using a similar process we obtain that the discrete phenomenological constraint $p_d$ associated to the finite difference map $\varphi$ is given by:
\[
	p_d(\hat{\bm{\Omega}}_k,\bm{\Gamma}_k,S_k,S_{k+1})=\frac{S_{k+1}-S_k}{h}-\frac{8\pi\mu a^3}{T(S_k)}\|\bm{\Omega}_k\| ^2.
\]
From Section \ref{sec:group-difference-map} and our choice for the discrete Lagrangian $\lag_d$, our variational integrator is given by the relations:
\begin{empheq}[left=\empheqlbrace]{align*}
	&(\mathrm{d}\operatorname{cay}_{h\hat{\bm{\Omega}}_k}^{-1})^*D_1\lag(\hat{\bm{\Omega}}_k,\bm{\Gamma}_k,S_k,S_{k+1})\\
	&\qquad =(\mathrm{d}\operatorname{cay}_{-h\hat{\bm{\Omega}}_{k-1}}^{-1})^*D_1\lag(\hat{\bm{\Omega}}_{k-1},\bm{\Gamma}_{k-1},S_{k-1},S_{k})
	-h\mathbf{J}\big(D_2\lag(\hat{\bm{\Omega}}_k,\bm{\Gamma}_k,S_k,S_{k+1})\big)\\
	&\qquad\;\;\;\;+f_d^{-}(\hat{\bm{\Omega}}_k,\bm{\Gamma}_k,S_k,S_{k+1})+f_d^{+}(\hat{\bm{\Omega}}_{k-1},\bm{\Gamma}_{k-1},S_{k-1},S_{k}),\\[2mm]
	& p_d(\hat{\bm{\Omega}}_k,\bm{\Gamma}_k,S_k,S_{k+1})=0,\\[2mm]
	& \bm{\Gamma}_{k+1}=\operatorname{cay}(-h\hat{\bm{\Omega}}_k)\bm{\Gamma}_k,
\end{empheq}
where the momentum map $\mathbf{J}:T^*M\to\so(3)^*$ is given by $\mathbf{J}(\mathbf{V},\mathbf{W})=[\hat{\bm{\mathbf{V}}},\hat{\bm{\mathbf{W}}}]$. Using the discrete Lagrangian \eqref{discrete_Lagr} and the expression
\[
(\mathrm{d}\operatorname{cay}_{\hat{\bm{\Omega}}}^{-1})^*({\hat{\bm{\Pi }}})=\mathrm{d}\operatorname{cay}_{-\hat{\bm{\Omega}}}^{-1}({\hat{\bm{\Pi }}})=\bigg(\Id+\frac{\hat{\bm{\Omega}}}{2}\bigg)\hat{\bm{\Pi}}\bigg(\Id-\frac{\hat{\bm{\Omega}}}{2}\bigg), 
\]
our variational integrator is explicitly given by:
\begin{empheq}[left=\empheqlbrace]{align}
	&\frac{\hat{\bm{\Pi}}_{k+1}-\hat{\bm{\Pi}}_{k}}{h} +  \frac{1}{2}\left([\hat{\bm{\Omega}}_{k+1},\hat{\bm{\Pi}}_{k+1}]+[\hat{\bm{\Omega}}_{k},\hat{\bm{\Pi}}_{k}]\right)-\frac{h}{4}(\hat{\bm{\Omega}}_{k+1}\hat{\bm{\Pi}}_{k+1}\hat{\bm{\Omega}}_{k+1}-\hat{\bm{\Omega}}_{k}\hat{\bm{\Pi}}_{k}\hat{\bm{\Omega}}_{k})\notag\\
	&\hspace{6.5cm}-mg\ell[\hat{\bm{\Gamma}}_{k+1},\hat{\bm{\chi}}]+4\pi\mu a^3( {\hat{\bm{\Omega}}_{k+1}+\hat{\bm{\Omega}}_k})=0,\label{eq:variational-integrator-1}\\
	& S_{k+1}=S_k+\frac{8\pi\mu a^3 h}{T(S_k)}\|\bm{\Omega}_k\| ^2 ,\label{eq:variational-integrator-2}\\
	& \bm{\Gamma}_{k+1}=\bigg(\Id+\frac{\hat{\bm{\Omega}}_k}{2}\bigg)^{-1}\bigg(\Id-\frac{\hat{\bm{\Omega}}_k}{2}\bigg)\bm{\Gamma}_{k}\label{eq:variational-integrator-3},
\end{empheq}
where $\bm{\Pi}_k= \mathbb{I}\bm{\Omega}_k$ is the discrete angular momentum. {Note that the first equation is written in $ \mathfrak{so}(3)$. However, it can be easily rewritten in $ \mathbb{R}  ^3 $ by using the definition \eqref{hat_map} of the hat map , as well as the formula $\hat{\bm{\Omega}}\hat{\bm{\Pi}}\hat{\bm{\Omega}}=\hat{ \mathbf{v} }$, for $ \mathbf{v} =\boldsymbol{\Omega} \times (\boldsymbol{\Pi} \times\boldsymbol{\Omega} )- |\boldsymbol{\Omega} | ^2\boldsymbol{\Pi} $.} These relations hold for $k\in\{0,\dots,N-1\}$. Given the input $\hat{\bm{\Omega}}_k$, $\hat{\bm{\Gamma}}_k$ and $S_k$, one step of the variational integrator outputs {$\hat{\bm{\Omega}}_{k+1}$, $\hat{\bm{\Gamma}}_{k+1}$ and $S_{k+1}$} as follows. Firstly notice that the new value $S_{k+1}$ can be computed from \eqref{eq:variational-integrator-2} whenever we want since it only depends on the previous value $\hat{\bm{\Omega}}_k$ of the body angular velocity and the previous value $S_k$ of the entropy. From \eqref{eq:variational-integrator-3}, the same is true for the new value $\hat{\bm{\Gamma}}_{k+1}$ of the advected parameter. However, this new value $\hat{\bm{\Gamma}}_{k+1}$ is needed to compute the new value $\hat{\bm{\Omega}}_{k+1}$ of the body angular velocity, as can be seen from \eqref{eq:variational-integrator-1}, which is a nonlinear equation in $\hat{\bm{\Omega}}_{k+1}$ that we solve using a Newton-Krylov method. Note that in the absence of thermal effects and torque, we recover the variational integrator presented in \citet[Section 4.1.2]{GMPMD2011}.

\begin{remark}
	Note that in this particular example the entropy equation \eqref{eq:variational-integrator-2} is totally decoupled from the momentum equation \eqref{eq:variational-integrator-1}. In order to have a fully coupled physical model, one could think of the top as changing its mass repartition as the temperature is changing (like when one boils an egg); this amounts to make the inertia tensor $\mathbb{I}$ depend on the entropy $S$. 
\end{remark}

The discrete total energy is defined by:
\[
	e_k=e(\hat{\bm{\Omega}}_k,\bm{\Gamma}_k,S_k)=\frac{1}{2}\mathbb{I}\bm{\Omega}_k\cdot\bm{\Omega}_k+mg\ell\bm{\Gamma}_k\cdot\bm{\chi}+U_\mathcal{B}(S_k).
\]

From Section \ref{sec:group-difference-map} we also obtain a discrete Kelvin-Noether theorem for our system. From \eqref{I_d}, \eqref{DKN} and the expressions \eqref{f_fr_discret} for the discrete friction forces \eqref{f_fr_discret}, we get a relation describing the rate of change of the discrete spatial angular momentum for $k\in\{1,\dots,N\}$:
\begin{equation}\label{DKN_example} 
	I_k-I_{k-1}=-8\pi\mu a^3h^2\langle\hat{\bm{\Gamma}}_k,\hat{\bm{\Omega}}_k\rangle,
\end{equation} 
where
\[
	I_k=\left\langle\hat{\bm{\Gamma}}_k,\mathrm{d}\operatorname{cay}_{ -h\hat{\bm{\Omega}}_k}^{-1}(h\hat{\bm{\Pi}}_k)\right\rangle
	=\left\langle R_k\Big[\mathrm{d}\operatorname{cay}_{ -h\hat{\bm{\Omega}}_k}^{-1}(h\hat{\bm{\Pi}}_k)\Big]R_k^{-1},\hat{\mathbf{e}}_3\right\rangle,
\]
since $ \hat{\boldsymbol{\Gamma}} _k = \widehat{R_k ^{-1} \mathbf{e} _3}=R_k^{-1}\hat{\mathbf{e}}_3 R_k$. {We note that the discrete Noether theorem \eqref{DKN_example} is a relation that approximates \eqref{KN_example} and which is exactly verified by the solution of the integrator \eqref{eq:variational-integrator-1}--\eqref{eq:variational-integrator-3}.}

\subsection{Numerical simulation}

The parameters for the numerical simulation are as follows: $h=\SI{0.1}{s}$, $a=\SI{0.05}{m}$, $\mu=\SI{0.1}{kg.m^{-1}.s^{-1}}$ (motor oil), $\rho=\SI{2700}{kg.m^{-3}}$ (aluminium), $M=\SI{26.981539e-3}{kg.mol^{-1}}$ (aluminium). The total mass of the ball is $m=\frac{4}{3}\pi a^3\rho$. We assume that the ball is made of two hemispheres, the upper one is plain and has a mass $m_1=0.6m$, the lower one is hollow and has a mass  $m_2=m-m_1$. With this choice one computes that the center of mass of the heavy top is $G=\big(0,0,\frac{3m_1a}{8m}-\frac{m_2a}{2m}\big)$ from which we compute $\ell$ and $\bm{\chi}$. We also compute the moment of inertia tensor of the ball by summing the inertia tensors of the two hemispheres:
\[
	\mathbb{I}=
	\begin{pmatrix}
		\frac{83a^2}{320}m_1  & 0 & 0\\
		0 & \frac{83a^2}{320}m_1  & 0\\
		0 & 0 & \frac{2a^2}{5}m_1 
	\end{pmatrix}
	+
	\begin{pmatrix}
		\frac{a^2}{12}m_2  & 0 & 0\\
		0 & \frac{a^2}{12}m_2  & 0\\
		0 & 0 & \frac{a^2}{3}m_2
	\end{pmatrix}.
\]
Concerning the initial conditions, we place the heavy top such that the plain hemisphere lies on the positive $z$ axis and the hollow hemisphere lies on the negative $z$ axis, $R_0=\left(\begin{smallmatrix}1 & 0 & 0\\ 0 & 0 & -1\\ 0 & 0 & 1\end{smallmatrix}\right)$, $\bm{\Omega}_0=(0,1,1)$, $\bm{\Gamma}_0=R_0^{-1}\mathbf{e}_3$, $T_0=\SI{300}{K}$, $S_0=\SI{0}{kg.m^2.s^{-2}.K^{-1}}$ and $N_0=\frac{m}{M}$ (will be constant during the simulation).

Our simulation yields the following trajectory $\ell R\bm{\chi}$ for the center of mass:

\begin{figure}[H]
	\centering
	\includegraphics[width=\textwidth]{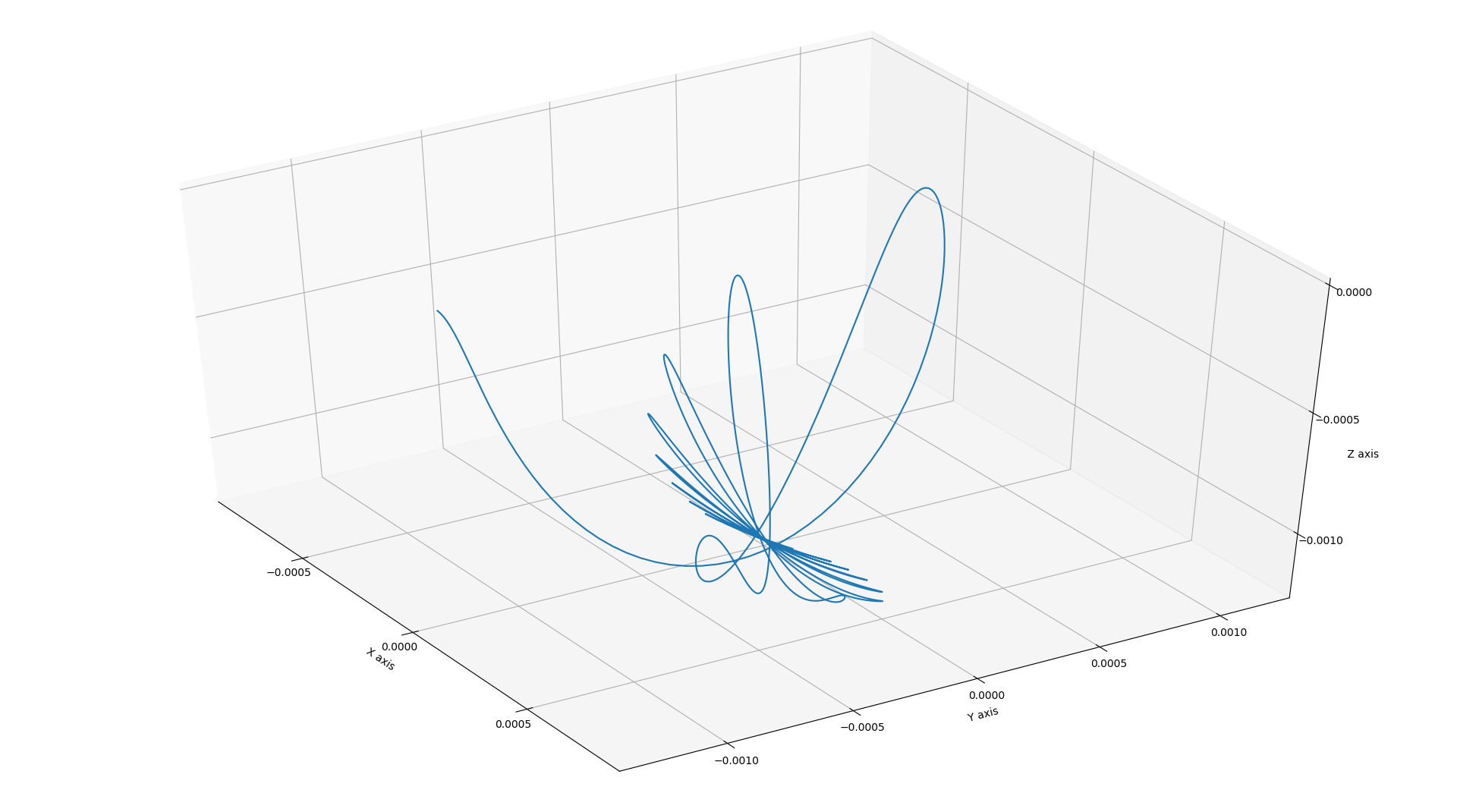}
	\caption{The trajectory of the center of mass of the heavy top.}
\end{figure}

{As expected, the angular velocity $\bm{\Omega}$ tends to zero and the center of mass oscillates around its limiting value $-\ell\mathbf{e}_3$, since the plain hemisphere is more massive.}

For the purpose of benchmarking, we used in parallel to our variational integrator the standard Runge-Kutta method of order 2, as without thermal effects our integrator can be seen to have order 2 \cite[Theorem 4.7.1]{BoRa2007}. The curves for the kinetic, potential and internal energies exhibit the following profiles:

\begin{figure}[H]
	\centering
	\includegraphics[width=\textwidth]{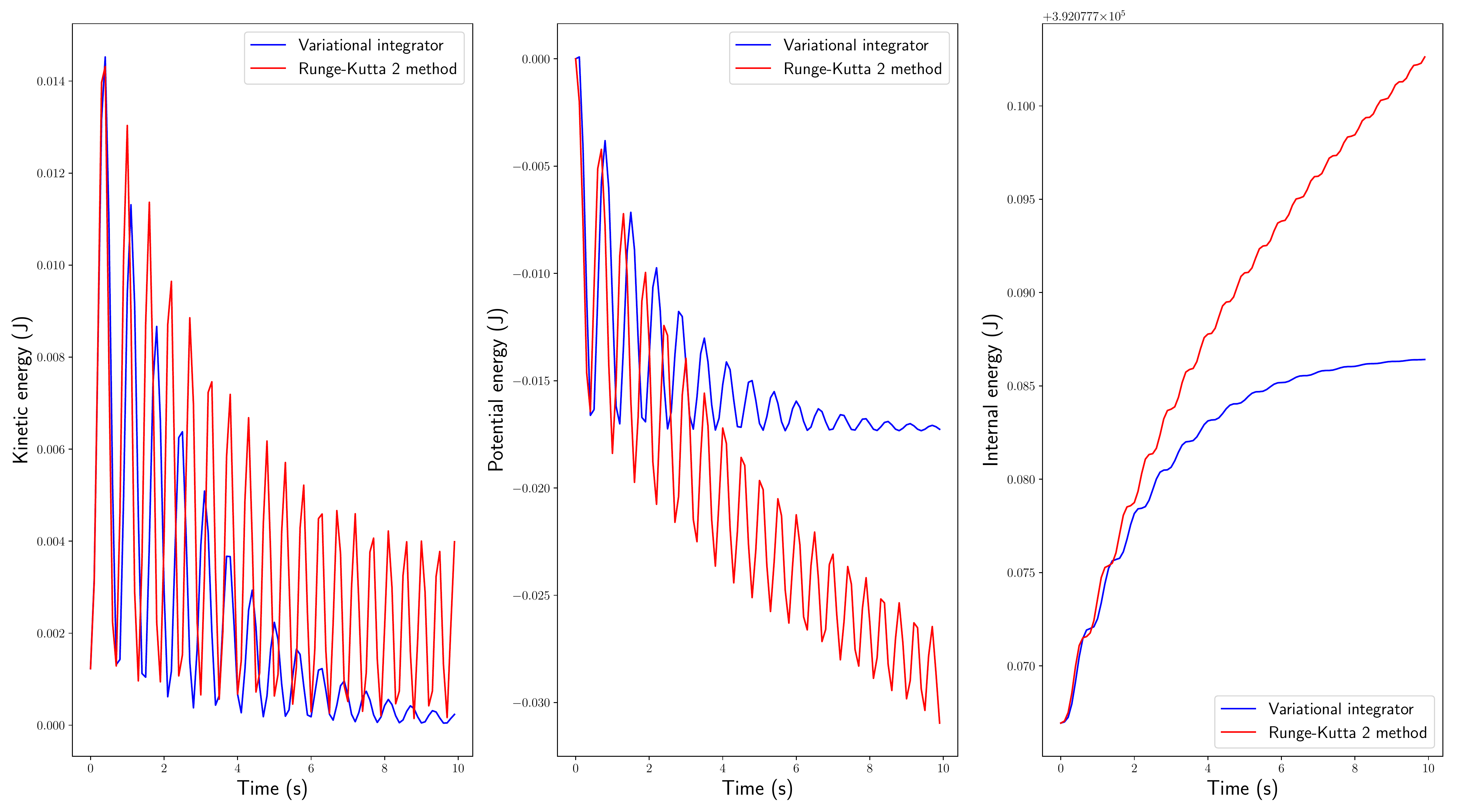}
	\label{fig:energies}
	\caption{The different energies of the system.}
\end{figure}

Note the difference between the height of the center of mass in both methods (which is proportional to the potential energy). The higher the viscosity, the less apparent is the difference, as higher viscosity means that the system is subject to more friction, and that its dynamics is less chaotic. The most interesting aspect is the behavior of the total energy of the system:

\begin{figure}[H]
	\centering
	\includegraphics[width=\textwidth]{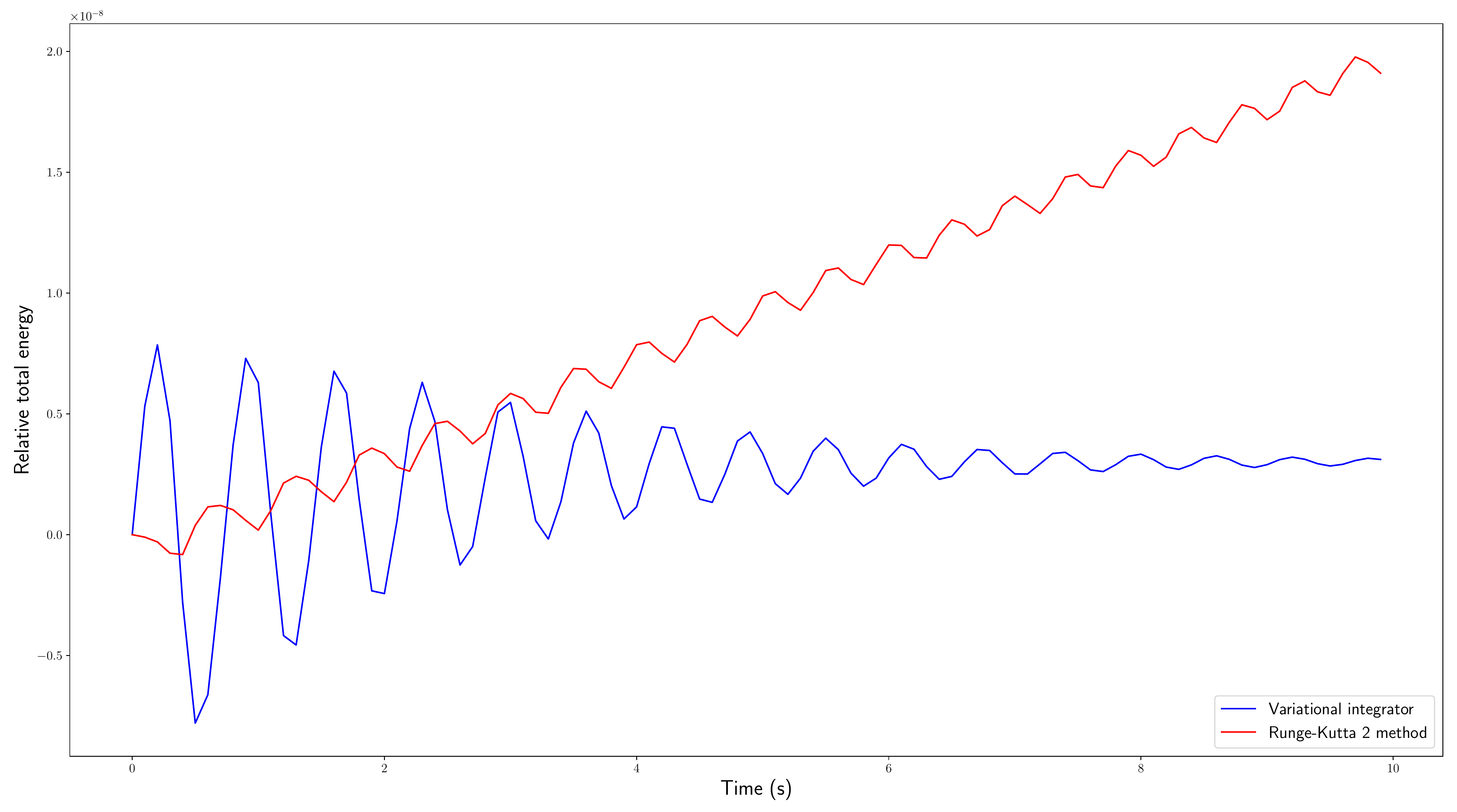}
	\label{fig:total-energy}
	\caption{The relative total energy of the system. While the Runge-Kutta 2 method yields an increase in the total energy, our variational integrator displays the usual oscillatory behaviour until the system stops moving, even with a large time step.}
\end{figure}

Remember from the previous section that the total energy of the system is constant (see the green line above). In a similar way to the variational integrators in Lagrangian mechanics, our integrator exhibits the oscillatory behavior around the true value of the total energy. Concerning the entropy of the system, it is in accordance with the second law of thermodynamics: the entropy increases as the system experiences an irreversible process. We tested several initial conditions and observed each time the expected oscillatory total energy behavior around the exact value, in accordance with the first law of thermodynamics. The variational integrator thus captures well the conversion of mechanical into thermal energy.

\bigbreak\noindent\textbf{Conclusion and outlook:} In this article we have presented the continuous and discrete variational formulations of simple thermodynamical systems on (finite dimensional) Lie groups. On the continuous side, we applied the variational formulation of \cite{GBYo2017a} to the case when the configuration manifold is a finite dimensional Lie group and, by assuming symmetries, we extended to the thermodynamical setting the well-known process of Euler-Poincar\'e reduction for mechanical systems on Lie groups. Based on these developments, and following  \cite{GBYo2017c}, we deduced a variational discretization for such thermodynamical systems, that extends earlier variational integrators for mechanical systems on Lie groups. We then illustrate the good behavior of the variational scheme on the example of a heavy top in a Stokes flow. This example only illustrates a simplified situation of the general variational setting that we developed in the paper. The next step is to leverage this integrator and apply it to more complicated settings such as fluids in the presence of irreversible processes (viscosity, heat conduction), which exhibit a complete coupling of the mechanical and thermal equations. In order to achieve this goal, the configuration space, which is an infinite-dimensional Lie group of diffeomorphisms, has to be discretized into a finite-dimensional one first. This can be done in two different ways at least: with the help of the \emph{sine-bracket approach} for two dimensional incompressible fluids on the torus, see \cite{Ze1991}, or with the help of \emph{discrete diffeomorphism groups} for incompressible, see \cite{PaMuToKaMaDe2011}, and compressible fluids, see \cite{BaGB2018}. We will also need to discretize the phenomenological constraint in this particular setting, which according to preliminary work, proves to be difficult. An important point of interest is that in the case of compressible fluids, the discrete evolution equations will be fully coupled.

\bigbreak\noindent\textbf{Acknowledgments:} The authors were financed by the ANR project GEOMFLUID (ANR-14-CE23-0002). We thank H. Yoshimura for very helpful comments. The first author thanks S. Shamekh for helpful discussions concerning the example.

\end{document}